\documentclass[a4paper,12pt]{article}

\usepackage[bbgreekl]{mathbbol}
\usepackage{mathrsfs}
\usepackage{graphicx}
\usepackage{amsmath}
\usepackage{amsfonts}
\usepackage{indentfirst}
\usepackage{amsthm}
\usepackage{amssymb}
\usepackage{algorithm, algorithmic}
\usepackage{xcolor}
\usepackage{geometry}
\usepackage{url}
\usepackage{setspace}
\usepackage{verbatim}
\usepackage[normalem]{ulem}
\usepackage{multirow}
\usepackage{booktabs}
\usepackage[page,toc,titletoc,title]{appendix}
\usepackage{bm}
\usepackage{cases}
\usepackage{subfigure}
\usepackage{cancel}

\newtheorem{theorem}{Theorem}[section]
\newtheorem{lemma}{Lemma}[section]
\newtheorem{remark}{Remark}[section]

\newcommand{\R}{\mathbb{R}}

\newcommand{\Z}{\mathbb{Z}}
\newcommand{\N}{\mathbb{N}}
\newcommand{\F}{\mathcal{G}}
\newcommand{\LL}{\mathscr{L}}
\newcommand{\mv}{\bar{\varphi}} 
\newcommand{\ml}{\bar{\lambda}} 
\newcommand{\ev}{z} 
\newcommand{\evt}{\xi} 
\newcommand{\tv}{\psi} 
\newcommand{\sad}{\bar{s}}
\newcommand{\ds}{k_{\sad}}
\newcommand{\ddelta}{k_{\zeta}}
\newcommand{\sA}{\sigma_{A}}
\newcommand{\sB}{\sigma_{B}}

\newcommand{\dv}{\varphi}

\newcommand{\tdmv}{\psi_h} 

\newcommand{\dd}{~{\rm d}}

\newcommand{\drp}{\gamma}
\newcommand{\Co}{C^1}    
\newcommand{\Cz}{C}

\newcommand{\U}{\mathscr{A}}  
\newcommand{\X}{X}  
\newcommand{\Y}{Y}

\title{Convergence of the Discrete Minimum Energy Path}

\author{
Xuanyu Liu\footnote{
    {\tt xyliu9535@mail.bnu.edu.cn}.
	School of Mathematical Sciences, Beijing Normal University, Beijing 100875, China.},
~ Huajie Chen\footnote{
Corresponding.
    {\tt chen.huajie@bnu.edu.cn}.
	School of Mathematical Sciences, Beijing Normal University, Beijing 100875, China.
    This author was partially supported by the National Natural Science Foundation of China (No. 12371431).}
~ and~ Christoph Ortner\footnote{
    {\tt ortner@math.ubc.ca}.
	Department of Mathematics, University of British Columbia, 1984 Mathematics Road, Vancouver, BC, Canada V6T 1Z2.}
}

\date{}

\begin{document}
\maketitle
 
\begin{abstract}
The minimum energy path (MEP) describes the reaction pathway between two stable configurations of a physical system, as well as the rate of reaction. 
The nudged elastic band (NEB) method is one of the most commonly used schemes to compute MEPs numerically. It approximates an MEP by a discrete set of configuration images, where the discretization size determines both computational cost and accuracy of the MEP approximation. 
We consider a discrete MEP to be a stationary state of the NEB method and prove an optimal convergence rate of the discrete MEP with respect to the number of images.
Numerical simulations for the transitions of prototypical model systems support the theory.
\end{abstract} 

\vskip 0.2cm


\section{Introduction}
\label{sec:introduction}
\setcounter{equation}{0}

In computational chemistry, reaction mechanisms are often represented by a reaction path on the potential energy landscape. Those paths determine the thermodynamic and kinetic properties of chemical reactions.
The minimum energy path (MEP) has been most widely used to represent the reaction path, giving the route that needs the least amount of work for the system to undergo the transition. 
The MEP connects two local minimizers and passes through one or more transition states (saddle points). The energy barrier and vibrational frequencies at the transition state(s) can be used to calculate the rate of the reaction~\cite{berglund2013kramers,hanggi1990reaction,vineyard1957frequency}.

The most widely used techniques for finding the MEP are the nudged elastic band (NEB) method \cite{2000climbingNEB,jonsson1998NEB} and the string method \cite{2002string,2007string,2013climbingstring}.
They both iteratively evolve a discrete path of images in projected steepest descent directions, while keeping a smooth distribution of the images along the path.
The number of images along the path determines the accuracy and efficiency of the simulation, especially when the energy functional is expensive to evaluate (for example, when the system is described by a quantum mechanical model).
This makes it particuarly important to obtain sharp error bounds of the MEP approximation with respect to the number of images.

Notwithstanding the importance of such a convergence analysis, there is limited work devoted to the analysis of the MEP.
In \cite{cameron2011string}, the authors gave a qualitative analysis of the evolution of the path by the string method from a dynamical systems perspective, characterizing the set of limiting curves and establishing conditions under which the limiting curve is indeed a MEP.
In \cite{2018LuskinStabString}, the authors proposed a uniform asymptotic stability of the MEP in the sense that any curve near the MEP can be arbitrarily close to it in the Hausdorff distance under the gradient decent dynamics, and showed that the string method initialized in a neighborhood of the MEP can converge to an arbitrarily small neighborhood of the limiting MEP. 
The notion of stability employed in these works cannot give an explicit relationship between the deviation to the MEP and its force, hence it may not be able to provide an explicit {\it error bound} with respect to the number of images. 

To prepare for a comprehensive and sharp numerical analysis of MEP approximation schemes, we began in \cite{MEPStab} to establish a new framework for studying the stability of MEPs. 
We showed that the displacement of a curve to the MEP can be directly bounded by its force under appropriate weighted function space norms. 
This result in particular directly implies the stability of MEPs under perturbations of the potential energy landscape. 
It also suggests a path towards studying the convergence of the discrete MEP via establishing an analogous stability for the discrete MEP. 

Discrete MEP stability does not follow in a straightforward manner from existing results and will be the main technical achievement of the present work. Once this is established we can then deduce sharp convergence rates of discrete MEPs to their continuous limits.
We will indeed obtain an optimal convergence rate under mild and natural technical assumptions on the limit MEP. More precisely, we will show that (see Theorem \ref{thm2:conv})
\begin{equation*}
\big\| \mv - \bar{\phi}_M \big\| \leq C M^{-1},
\end{equation*}
where $\mv$ and $\bar{\phi}_M$ are the MEP and the discrete MEP respectively, $\|\cdot\|$ denotes a discrete $C^1$-type-norm, and $M$ is the number of images used to discretize the path.
Although we focus on the NEB framework in this work, the convergence result could be generalized to other methods that find the MEP, for instance the string method (see Remark \ref{rem:conv_string}).
To our best knowledge, this is the first result that gives an explicit convergence rate for the MEP with respect to the discretization parameters.

\vskip 0.2cm

\textbf{Outline.} 
The rest of this paper is organized as follows. 
In Section \ref{sec:results}, we present the main results of this paper, including the convergence result and numerical experiments on prototypical example systems.
In Section \ref{sec:conclusion}, we give some conclusions.
In Section \ref{sec:dmep}, we present detailed proofs for the convergence results.

\vskip 0.2cm

\textbf{Notation.}  
We denote the Euclidean norm and $\ell^\infty$-norm of a vector, respectively, by $\lvert \cdot \rvert$ and $|\cdot|_{\infty}$.
The symbol $\|\cdot\|_{\infty}$ is used to denote the matrix norm induced by $|\cdot|_{\infty}$.
For $b>a$, we denote by $C\left([a,b];\R^{N}\right)$ the space of continuous curves in the configuration space $\R^{N}$, equipped with the norm $\lVert\varphi\rVert_{C([a,b];\R^{N})} := \sup_{\alpha\in[a,b]} |\varphi(\alpha)|_{\infty}$; 
and $C^1\left([a,b];\R^{N}\right)$ the space of continuously differentiable curves with the norm $\lVert\varphi\rVert_{\Co([a,b];\R^{N})} := \lVert\varphi\rVert_{C([a,b];\R^{N})} + \lVert \varphi' \rVert_{C([a,b];\R^{N})}$. 
Let $X$ and $Y$ be Banach spaces equipped with norms $\|\cdot\|_X$ and $\|\cdot\|_Y$.
We will denote by $\LL(X,Y)$ the Banach space of all linear bounded operators from $X$ to $Y$ with the operator norm $\lVert\cdot\rVert_{\LL(X,Y)}$.
For a functional $\mathscr{F}\in C^1(X)$ (in the sense of Fr\'echet), we will denote its first variation by $\delta\mathscr{F}(x)v$ with $v\in X$. 

We will use $C$ to denote a generic positive constant that may change from one line to the next.
In the error estimates, $C$ will always be independent on the discretization parameters or the choice of test functions. The dependencies of $C$ on model parameters (in our context, the energy landscape) will normally be clear from the context or stated explicitly.

\section{Main results}
\label{sec:results}
\setcounter{equation}{0}

\subsection{Minimum energy path}
\label{sec:mep}

Let $E :\R^{N}\rightarrow \R$ be a potential energy functional with $N\in\N$ the dimensionality of the configuration space, which could encode atomic positions, structure of crystal lattices, and many other examples.
Throughout this paper we assume that $E \in C^4(\R^N)$.
This regularity assumption is required because our analysis requires detailed control of perturbations of the eigenvalues and eigenvectors of the Hessian $\nabla^2 E$. 

Given an energy functional $E$, we call $y\in\R^N$ a critical point if $\nabla E(y) = 0$. 
We call a critical point $y$ a strong local minimizer if the Hessian $\nabla^2 E(y)\in\R^{N\times N}$ is positive definite, and an index-1 saddle point if $\nabla^2 E(y)$ has exactly one negative eigenvalue while all the other eigenvalues are strictly positive. For the sake of brevity we will omit the qualifiers ``strong" and ``index-1" and simply say ``local minimizer" and ``saddle point".
We assume throughout that $E$ has at least two local minimizers on the energy landscape denoted, respectively, by $y^A_{\rm M}\in\R^N$ and $y^B_{\rm M}\in\R^N$. 

A minimum energy path (MEP) is a curve $\varphi\in C^1\big([0, 1];\R^{N}\big)$ connecting $y^A_{\rm M}$ and $y^B_{\rm M}$, whose tangent is everywhere parallel to the gradient except at the critical points. 
To give a rigorous definition of the MEP, we first define two projection operators $P_v~,P_v^{\perp}:\R^N\rightarrow\R^N$ with a given a vector $v\in\R^N \setminus \{0\}$, 
by $\displaystyle P_v := v v^{\rm T} /|v|^2$ and $ P^\perp_{v} := I - P_v$, where $I\in\R^{N\times N}$ is the identity matrix.
Define the admissible class for 
``regular" curves connecting the minimizers $y^A_{\rm M}$ and $y^B_{\rm M}$ as
\begin{align*}
\U := \Big\{ \varphi \in C^1\big( [0,1];\R^{N} \big) :~  \varphi(0) = y^A_{\rm M},~\varphi(1) = y^B_{\rm M},~\varphi'(\alpha) \neq 0 ~~\forall~\alpha\in[0,1] \Big\} .
\end{align*}
Then a MEP connecting the minimizers $y^A_{\rm M}$ and $y^B_{\rm M}$ is a solution of the following problem: Find $ \varphi \in \U $ such that
\begin{subequations} 
\label{mep}
	\begin{numcases}{}
	\label{mep1}\hspace{5pt}
	P^\perp_{\varphi'(\alpha)}\nabla E \big( \varphi(\alpha) \big) = {0} 	\hspace{5ex}~\forall~\alpha \in [0,1], \\[1ex]
	\label{mep2}\hspace{5pt}
	|\varphi'(\alpha)| - L(\varphi)= 0 \hspace{7ex}~\forall~\alpha \in [0,1],
	\end{numcases}
\end{subequations}
where $L : C^1([0,1];\R^N) \rightarrow \R$ is the length operator given by $L(\varphi) := \int_0^1 \lvert \varphi'(s) \rvert \dd s$.
Equation \eqref{mep2} enforces the curve to be parameterized by normalized arc length, which removes the redundancy due to re-parameterization. With this additional constraint one intuitively expects that the MEP is locally unique, at least under some further natural assumptions that we now formalize: 

If $\mv$ is a solution of \eqref{mep}, since $y^A_{\rm M}$ and $y^B_{\rm M}$ are local minimizers, there exists an $\sad\in(0,1)$ with $y_{\rm S}=\mv(\sad)\in\R^N$ such that the energy $E(y_{\rm S})$ reaches the maximum along the MEP. 
This implies that $\nabla E(y_{\rm S})$ vanishes in the tangent direction $\mv'(\sad)$ and thus it is a critical point. 
Since the energy $E(y_{\rm S})$ is a maximum along the tangent $\mv'(\sad)$, we may generically expect that the Hessian $\nabla^2E(y_{\rm S})$ has at least one negative eigenvalue.
For the sake of simplicity of the analysis, we will assume throughout this paper that  
\begin{flushleft}
\label{assump:index1saddle}
{\bf (A)}~
$y_{\rm S}$ is {\em the only} critical point between $y^A_{\rm M}$ and $y^B_{\rm M}$ along the MEP $\mv$. Moreover, $y^A_{\rm M}, y^B_{\rm M}$ are strong minimizers, while $y_{\rm S} = \mv(\sad)$ is an {\em index-1 saddle}.
\end{flushleft}

Although ${\bf (A)}$ is natural and will be satisfied by {\em many} (if not most) MEPs one encounters in practice, there are also cases where this assumption fails. For example, examples are given in \cite{MAP} where there is more than one {\em index-1 saddle} along an MEP.
Our theory can be generalized to these cases by adjusting the formulations, provided that all critical points along the MEP satisfy certain stability conditions. 

We observe by a direct calculation (see Lemma \ref{lemma:lambda} for details) that, if $\mv\in C^2\big( [0,1];\R^{N} \big)$ solves \eqref{mep}, then $\mv'(0)$, $\mv'(\sad)$ and $\mv'(1)$ are eigenvectors of, respectively the Hessians $\nabla^2 E(y^A_{\rm M})$, $\nabla^2 E(y_{\rm S})$ and $\nabla^2 E(y^B_{\rm M})$.
This implies that the MEP has to go through the critical points in the direction of some eigenvector of the corresponding Hessian.
The following assumption formalizes the requirement that, to highest order, there is a unique optimal path to exit the energy minimizer.

\begin{flushleft}
\label{assump:simplelowest}
{\bf (B)}~
Let $\sA$ and $\sB$ denote the eigenvalues associated, respectively, with the eigenvectors $\mv'(0),\mv'(1)$. We assume (i) that they are the lowest eigenvalues of $\nabla^2 E(y^A_{\rm M})$ and $\nabla^2 E(y^B_{\rm M})$; and (ii) that they are simple.
\end{flushleft}

Assumption {\bf (B)} does limit the range of MEPs that are covered by our analysis. It is a crucial ingredient to obtain a strong linearized stability result and hence optimal convergence rates for the discrete MEP. We will observe in numerical tests that only weaker results can be expected when {\bf (B)} is not satisfied.

\subsection{The discrete MEP}
\label{sec:dMEP}

We now describe the discretization of a continuous MEP, by using a discrete path of images connecting the two minimizers. 
Let $M+1\in\Z_+$ be the number of images and $h:=\frac{1}{M}$ be the corresponding meshsize.
Let
\begin{align*}
\U_h := \Big\{ \phi_h := \{\phi_{h,k}\}_{k=0}^{M} \in (\R^N)^{M+1} :~ \phi_{h,0}=y^A_{\rm M}, ~\phi_{h,M}=y^B_{\rm M} ,
~{\rm and}~ \phi_{h,j}\neq\phi_{h,k}~\forall~j\neq k \Big\} 
\end{align*}
be the space of admissible discrete paths connecting the two minimizers $y^A_{\rm M}$ and $y^B_{\rm M}$ with $M+1$ images.
For a discrete $\phi_h\in\U_h$, we define the following {\it approximate tangent} at the images through an upwind scheme
\begin{align}
\label{eq:upwind}
(D_h\phi_h)_k := \left\{ 
\begin{array}{ll}
\displaystyle
\frac{\phi_{h,k+1} - \phi_{h,k}}{h}  
& {\rm if} ~ k = 0 ~ {\rm or} ~E(\phi_{h,k+1}) > E(\phi_{h,k}) > E(\phi_{h,k-1}),
\\[1ex]
\displaystyle
\frac{\phi_{h,k} - \phi_{h,k-1}}{h}
& {\rm if} ~  k = M ~ {\rm or}~ E(\phi_{h,k-1}) > E(\phi_{h,k}) >E(\phi_{h,k+1}),
\\[1ex]
\displaystyle
\frac{\phi_{h,k+1} - \phi_{h,k-1}}{2h}    
& {\rm otherwise}.
\end{array}\right.  
\end{align}
We further define the normalized tangent direction as
\begin{align}
\nonumber
& \hat{\tau}_k := \frac{(D_h\phi_h)_k}{|(D_h\phi_h)_k|} .
\end{align}

With a given discretization $\U_h$, the NEB method finds the MEP by evolving the discrete path under a driving force: Find $\phi_h \in\Co\big([0,+\infty);\U_h \big)$, such that
\begin{equation}
\label{eq:NEB_evolution}
\left \{~
\begin{aligned}
    \dot{\phi}_h(t) &= \F_h \big( \phi_h(t) \big) \qquad \forall~t>0 ,
    \\[1ex]
    \phi_h(0) &= \phi^0_h 
\end{aligned} 
\right.
\end{equation}
where $\phi^0_{h}\in\U_h$ is an initial discrete path.
Here the driving force $\F_h:\U_h \rightarrow (\R^N)^{M+1}$ is defined by 
\begin{equation}
\label{eq:NEBf}
\big( \F_h (\phi_h) \big)_k := \left\{~
\begin{aligned}
& 0 \hskip 2.5cm {\rm for}~k=0,M,
\\[1ex]
& - P_{\hat{\tau}_k}^\perp \nabla E(\phi_{h,k}) + c h^{-1} \Big( | \phi_{h,k}-\phi_{h,k+1} | - | \phi_{h,k}-\phi_{h,k-1} | \Big) \hat{\tau}_k 
\\[1ex]
& ~~ \hskip 2.5cm {\rm for}~k=1,\cdots,M-1,
\end{aligned}
\right.
\end{equation}
where $c$ is a preset parameter specifying the balance between the tangential force and the projected downward pull from the curvature of the potential barrier \cite{jonsson1998NEB}.
In \eqref{eq:NEBf}, the force perpendicular to the tangent $\hat{\tau}_k$ is to push the path towards the MEP, while the force parallel to the tangent $\hat{\tau}_k$ is to enforce an approximate equidistribution of the images along the path.
Let $\bar{\phi}_h\in\U_h$ be a stationary solution of the NEB evolution equation \eqref{eq:NEB_evolution}.
Then $\bar{\phi}_h$ solves the equation
\begin{eqnarray}
\label{dMEP:NEB}
\F_h (\bar{\phi}_h) = {0},
\end{eqnarray}
and we will call it a discrete MEP.

It is intuitive to expect that the NEB evolution path converges to the MEP in the limit as $t \to \infty$ and $h \to 0$. 
In the present paper, we focus on the discretization error of the curve, i.e. the limit $h \to 0$, while the convergence of the evolution, i.e., the limit $t \to \infty$, will be considered in a separate work.

\subsection{Convergence of the discrete MEP}
\label{sec:convergence:dMEP}

The following theorem is the main result of this paper, stating the convergence rate of a sequence of discrete MEPs as $M \to \infty$. 
The proof of this theorem is given in Section~\ref{sec:dmep}.
We not only consider the convergence of the path (under some discrete $C^1$-norm), but also the convergence of the approximate energy barrier.
For a given path $\varphi\in\U$, its energy barrier can be written as 
\begin{equation*}
    \Delta E(\varphi) := \sup_{\alpha\in[0,1]} E\big(\varphi(\alpha)\big) - E\big(\varphi(0)\big).
\end{equation*}
Correspondingly, we can define the energy barrier of a discrete path $\phi_h\in\U_h$ by 
\begin{equation*}
    \Delta_h E(\phi_h) := \max_{0\leq k\leq M} E(\phi_{h,k}) - E(\phi_{h,0}) .
\end{equation*}

\begin{theorem}
\label{thm2:conv}
Let $\mv\in C^3\big( [0,1];\R^{N} \big)$ be a solution of \eqref{mep} satisfying Assumptions {\bf (A)} and {\bf (B)}.
Then, for sufficiently small $h$ (equivalently, sufficiently large $M$), there exists a solution $\bar{\phi}_h$ of \eqref{dMEP:NEB} such that
\begin{align}
\label{est2:dismep}
\max_{0\leq k\leq M} \big| (D_h\bar{\phi}_h)_k - \mv'(k h) \big| + \max_{0\leq k\leq M} \big| \bar{\phi}_{h,k} - \mv(k h) \big| 
& \leq C_{\rm p} h
\qquad{\rm and}
\\[1ex]
\label{est:Eb}
\big|\Delta E(\mv) - \Delta_h E(\bar{\phi}_h)\big| 
& \leq C_{\rm e} h^2 ,
\end{align}
where $C_{\rm p}$ and $C_{\rm e}$ are positive constants depending only on $E$ and $\mv$.
\end{theorem}

\begin{remark}
The $\mathcal{O}(h)$ convergence rate in \eqref{est2:dismep} is optimal, since we only consider approximation of the tangent by first-order finite difference scheme \eqref{eq:upwind}.
\end{remark}

\begin{remark} 
\label{rem:conv_string}
Theorem \ref{thm2:conv} gives the convergence of the discrete MEP obtained by the NEB method, with the driving force $\F_h$ given in \eqref{eq:NEBf}.
It is likely that our analysis can be generalized to the discrete MEP obtained from other methods, particularly the string method.
For the string method, the discrete path $\phi_h$ is evolved by the driving force
\begin{equation*}
\big(\F_h^{\rm string}(\phi_h)\big)_k := - P_{(D_h\phi_h)_k}^\perp \nabla E(\phi_{h,k}) 
\qquad k=1,\cdots,M-1 ,
\end{equation*}
together with a redistribution of the images at each step during the evolution \cite{2007string}. 
Then the corresponding discrete MEP is defined by the stationary state of the equation $\F_h^{\rm string}(\bar{\phi}_h)=0$, such that the images on $\bar{\phi}_h$ are equally distributed along the path.
This gives the same path as that obtained by the NEB method.
However, most practical implementations of the string method employ cubic spline interpolations to obtain the tangent and additional work would be required to generalize our analysis.
\end{remark}

\begin{remark}
\label{remark:idea}
The critical ingredient in the proof of Theorem \ref{thm2:conv} will be the stability of the discrete MEP, or rather, stability of the discrete MEP equation. However, our strategy will {\em not} be to prove stability of the operator $\mathcal{G}_h$ defined in \eqref{eq:NEBf}. In Section~\ref{sec:stab_dmep} we will first reformulate the discrete MEP equation in terms of a carefully rescaled operator, for which we are then able to prove the stability in Theorem~\ref{le:stab_dmep}, which is our main new technical result.
\end{remark}

\subsection{Numerical experiments}
\label{sec:numerics}
We perform a range of numerical tests on model systems to support and extend our theoretical results.
All simulations are implemented in open-source {\tt Julia} packages {\tt JuLIP.jl} \cite{JuLIP}.
%
To test the decay of the numerical errors, results obtained by a very fine discretization with $1000$ images are used as the reference MEP. 

To support the theory developed in this paper, we test the convergence in the discrete $C^1$-norm error (see \eqref{est2:dismep}) and of the energy barrier error (see \eqref{est:Eb}) of the discrete MEP.
For completeness, we also test the convergence in a discrete $C$-norm (maximum norm). 
These three error measures are, respectively, denoted by
\begin{align*}
& e_{h,C^1} := \max_{0\leq k\leq M} \big| (D_h\bar{\phi}_h)_k - \mv'(k h) \big| + \max_{0\leq k\leq M} \big| \bar{\phi}_{h,k} - \mv(k h) \big| ,
\\[1ex]
& e_{h,E_{\rm b}} := \big|\Delta E(\mv) - \Delta_h E(\bar{\phi}_h)\big|
\qquad{\rm and}
\\[1ex]
& e_{h,C} := \max_{0\leq k\leq M} \big| \bar{\phi}_{h,k} - \mv(k h) \big|.
\end{align*}

\vskip 0.2cm

{\bf Example 1.}
Consider a toy model with the potential $E:\R\times\R_+\rightarrow\R$ with
\begin{equation*}
E(x,y) = \alpha (1-x^2-y^2)^2+ \frac{y^2}{x^2+y^2} \qquad{\rm with} ~ y\geq 0 ,
\end{equation*}
where the parameters $\alpha$ is chosen by three cases $1,~\frac{1}{4}$ and $\frac{1}{8}$.
In the first case $\alpha=1$, the eigenvalues of $\nabla^2 E(y^A_{\rm M})$ are 2 and 8, and $\sA = 2$ is simple and the lowest eigenvalue. 
In the second case $\alpha=\frac{1}{4}$, $\nabla^2 E_1(y^A_{\rm M})$ has an eigenvalue 2 with multiplicity 2, so $\sA = 2$ is the lowest but a degenerated eigenvalue of $\nabla^2E_1(y^A_{\rm M})$.
In the last case $\alpha=\frac{1}{8}$, the eigenvalues of $\nabla^2 E_2(y^A_{\rm M})$ are 2 and 1, in which $\sA = 2$ is simple but not the lowest eigenvalue. 
We see that the condition {\bf (B)} is satisfied by the first case, but not for the others.

The contour lines of the energy landscapes are shown in Figure \ref{fig:ex1:contour}, \ref{fig:ex1:contour-3} and \ref{fig:ex1:contour-2}, respectively.
We see that in all cases there are two minimizers located at $(-1,0)$ and $(1,0)$.
The exact MEPs are all explicitly known as the upper branch of the unit circle, from which we can calculate the numerical errors.
We present the decay of the numerical errors for the three cases in Figure \ref{fig:ex1:convergence}, \ref{fig:ex1:convergence-3} and \ref{fig:ex1:convergence-2}, respectively.

For the first case, we observe that the convergence rate of the discrete $C^1$-norm error is $\mathcal{O}(h)$ and that of the energy barrier error is $\mathcal{O}(h^2)$, which agrees with our theory. 
We further observe that the discrete $C$-norm error decays with the same rate as the discrete $C^1$-norm error, i.e., there is no improved rate when employing a weaker norm. This suggests that the spaces we employed in the analysis are in fact natural.

In the last two cases, we see that the discrete MEP also converges in discrete $C^1$-norm, but the convergence rates are significantly slower. 
This indicates that the condition {\bf (B)} is indeed necessary for the convergence result of Theorem \ref{thm2:conv}.
However, maybe surprisingly, we find that the discrete $C$-norm error and the energy barrier error still decay as $O(h)$ and $O(h^2)$, respectively. 
The methods of our present work are insufficient to explain these effects; obtaining sharp rates in the linearly unstable regime will require further work.

\begin{figure}[htbp!]
\centering
\begin{minipage}[t]{0.48\textwidth}
\centering
\includegraphics[width=7.5cm]{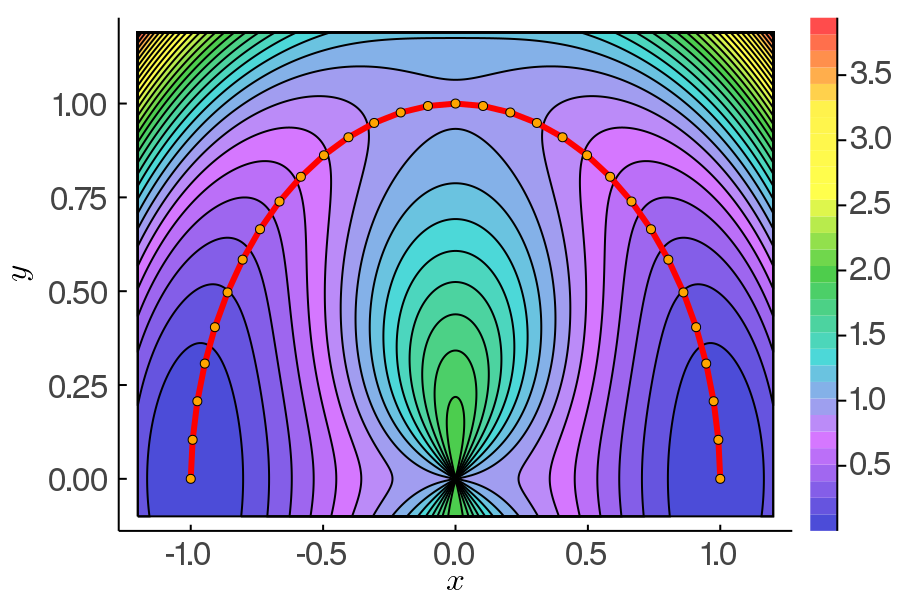}
\caption{(Example 1) The contour lines of the energy landscape  for $E$, with two minimizers and the MEP (red line).}
\label{fig:ex1:contour}
\end{minipage}
\hskip 0.5cm
\begin{minipage}[t]{0.48\textwidth}
\centering
\includegraphics[width=8.0cm]{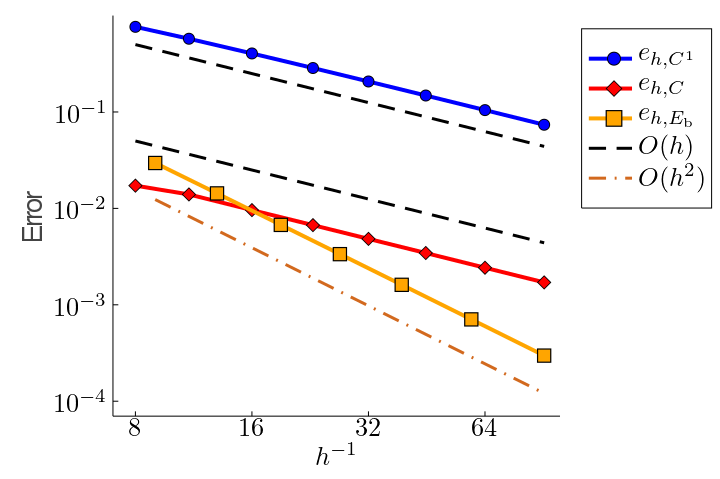}
\caption{(Example 1) The convergence of the discrete MEP ($\sigma_A$ is simple and lowest eigenvalue).
}
\label{fig:ex1:convergence}
\end{minipage}
\vskip 0.6cm
\begin{minipage}[t]{0.48\textwidth}
\centering
\includegraphics[width=7.5cm]{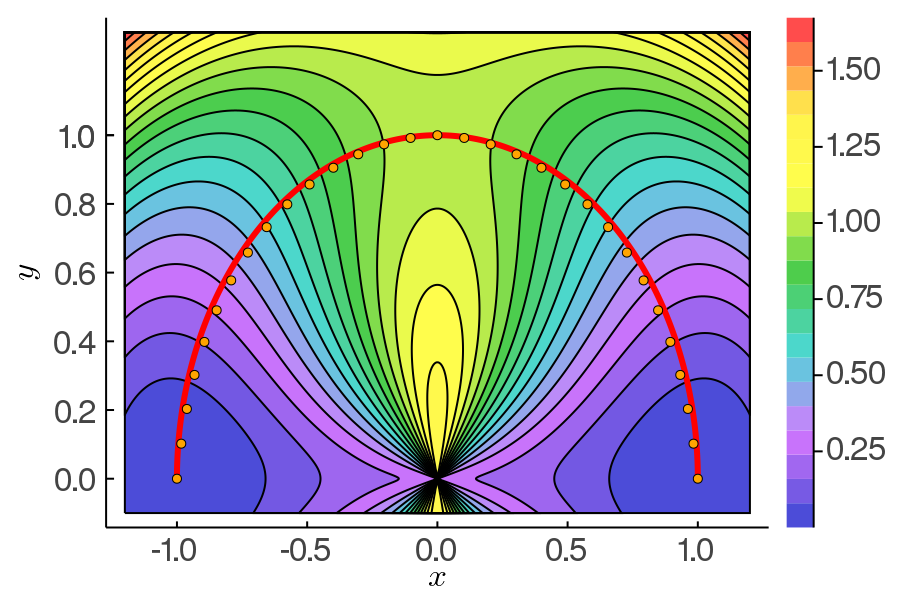}
\caption{(Example 1) The contour lines of the energy landscape for $E_1$, with two minimizers and the MEP (red line).}
\label{fig:ex1:contour-3}
\end{minipage}
\hskip 0.5cm
\begin{minipage}[t]{0.48\textwidth}
\centering
\includegraphics[width=7.2cm]{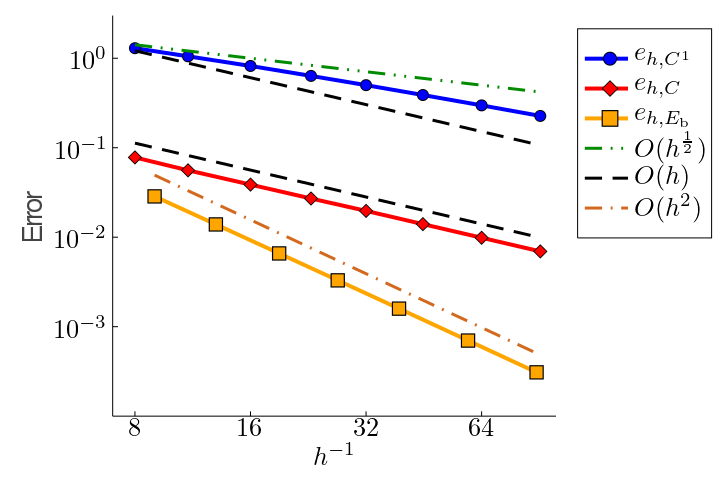}
\caption{(Example 1) The convergence of the discrete MEP ($\sigma_A$ is a degenerated eigenvalue).}
\label{fig:ex1:convergence-3}
\end{minipage}
\vskip 0.6cm
\begin{minipage}[t]{0.48\textwidth}
\centering
\includegraphics[width=7.5cm]{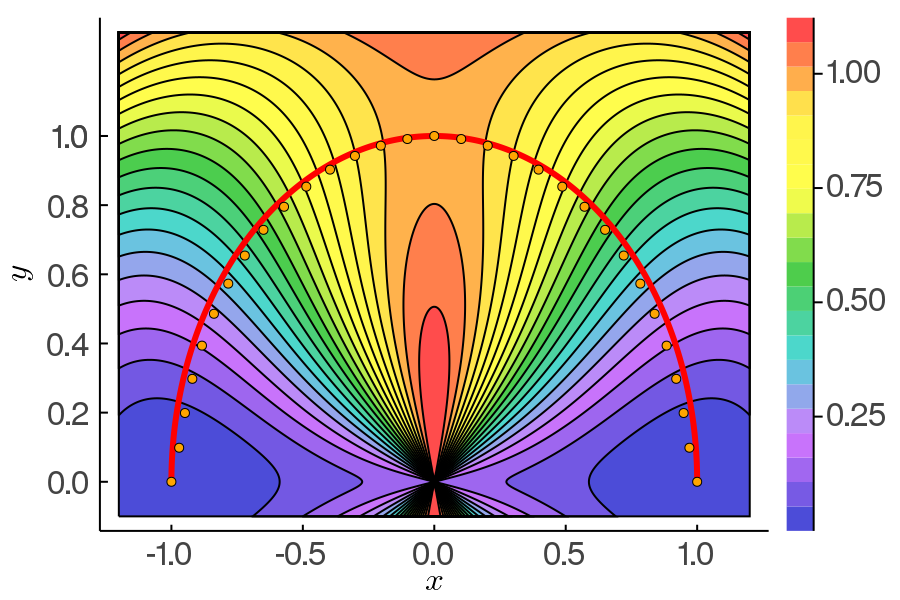}
\caption{(Example 1) The contour lines of the energy landscape for $E_2$, with two minimizers and the MEP (red line).}
\label{fig:ex1:contour-2}
\end{minipage}
\hskip 0.5cm
\begin{minipage}[t]{0.48\textwidth}
\centering
\includegraphics[width=7.2cm]{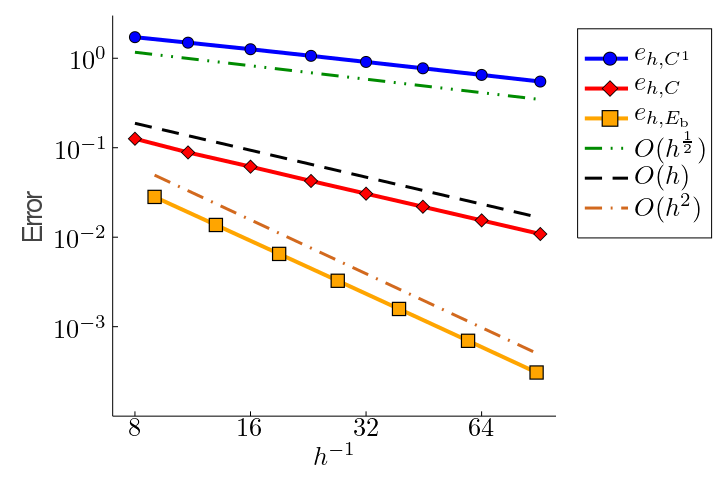}
\caption{(Example 1) The convergence of the discrete MEP ($\sigma_A$ is not the lowest eigenvalue).}
\label{fig:ex1:convergence-2}
\end{minipage}
\end{figure}

\vskip 0.2cm

{\bf Example 2.}
Next, we consider the Muller potential \cite{muller1979MEP} $E:\R\times\R \rightarrow \R$ with
\begin{equation*}
    E(x,y) = \sum_{k=1}^{4} T_k \exp\Big( a_k (x - \bar{x}_k)^2+b_k(x-\bar{x}_k)(y-\bar{y}_k) + c_k(y-\bar{y}_k)^2 \Big) , 
\end{equation*}
where the parameters are $T=(-200,-100,-170,-15)^{\rm T}$, $a=(-1,-1,-6.5,0.7)^{\rm T}$, $b=(0,0,11,0.6)^{\rm T}$, $c=(-10,-10,-6.5,0.7)^{\rm T}$, $\bar{x}=(1,0,-0.5,-1)^{\rm T}$ and $\bar{y}=(0,0.5,1.5,1)^{\rm T}$.
We present the contour lines of energy landscape in Figure \ref{fig:ex2:contour} and the decay of numerical errors in Figure \ref{fig:ex2:convergence}. 
We again observe an $\mathcal{O}(h)$ decay for the discrete $C^1$-norm and $C$-norm errors, and $\mathcal{O}(h^2)$ for the energy barrier error.

\begin{figure}[htbp!]
\centering
\begin{minipage}[t]{0.48\textwidth}
\centering
\includegraphics[width=7.3cm]{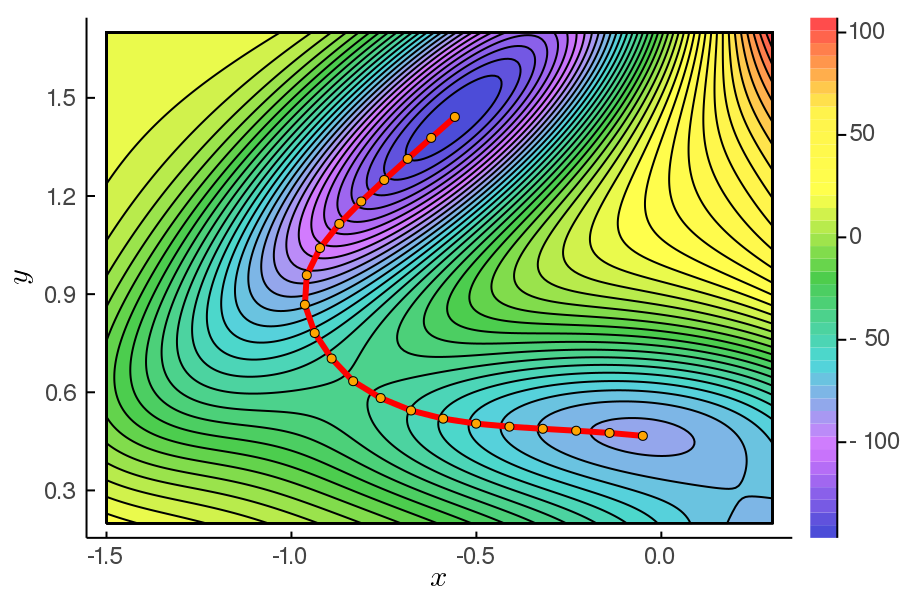}
\caption{(Example 2) The contour lines of the energy landscape, with two minimizers and the MEP (red line).}
\label{fig:ex2:contour}
\end{minipage}
\hskip 0.5cm
\begin{minipage}[t]{0.48\textwidth}
\centering
\includegraphics[width=8.0cm]{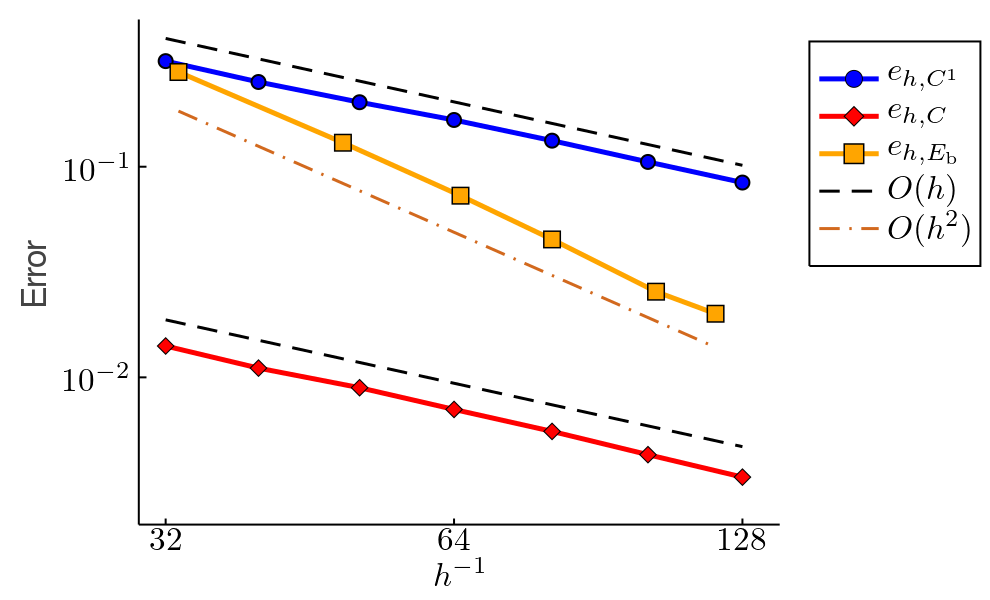}
\caption{(Example 2) Convergence of the discrete MEP.
}
\label{fig:ex2:convergence}
\end{minipage}
\end{figure}
\begin{figure}[htb!]
\centering
\begin{minipage}[t]{0.48\textwidth}
\centering
\includegraphics[width=7.5cm]{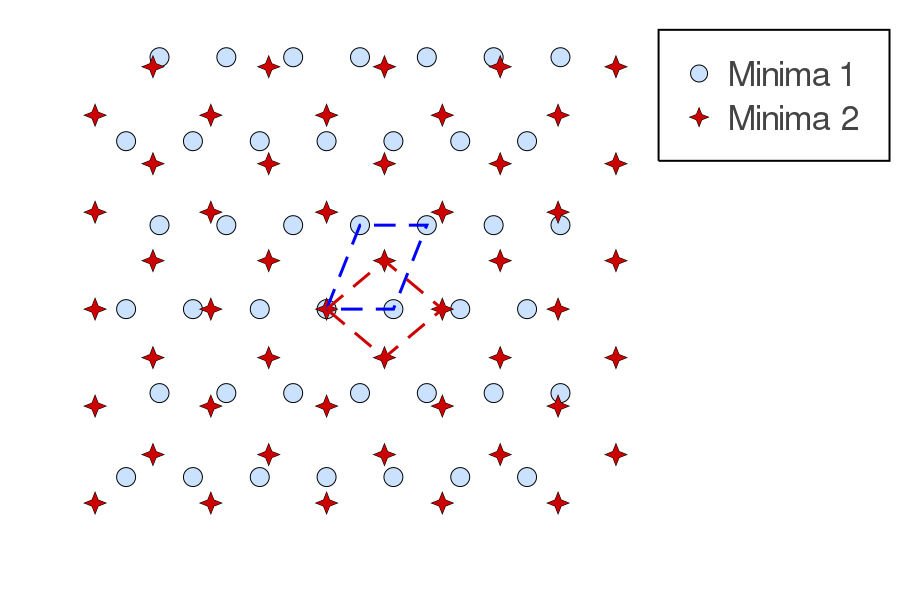}
\caption{(Example 3) The configurations of two minimizers.}
\label{fig:ex3:cell}
\end{minipage}
\hskip 0.5cm
\begin{minipage}[t]{0.48\textwidth}
\centering
\includegraphics[width=7.5cm]{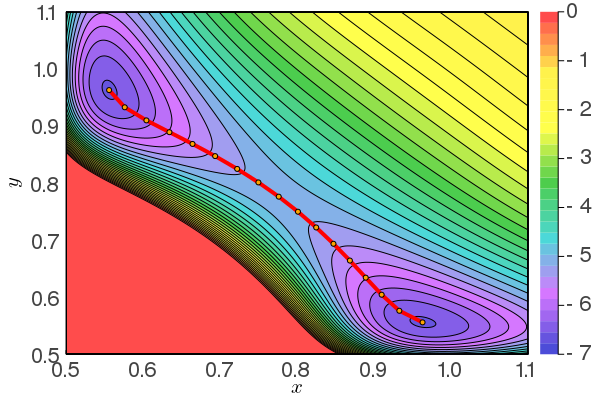}
\caption{(Example 3) The contour lines of energy landscape, with two minimizers and the MEP (red line).
Since the energy diverges when the two atoms collide, we only show contour lines with energy less than $0$.
}
\label{fig:ex3:contour}
\end{minipage}
\vskip 0.6cm
\begin{minipage}[t]{0.48\textwidth}
\centering
\includegraphics[width=7.0cm]{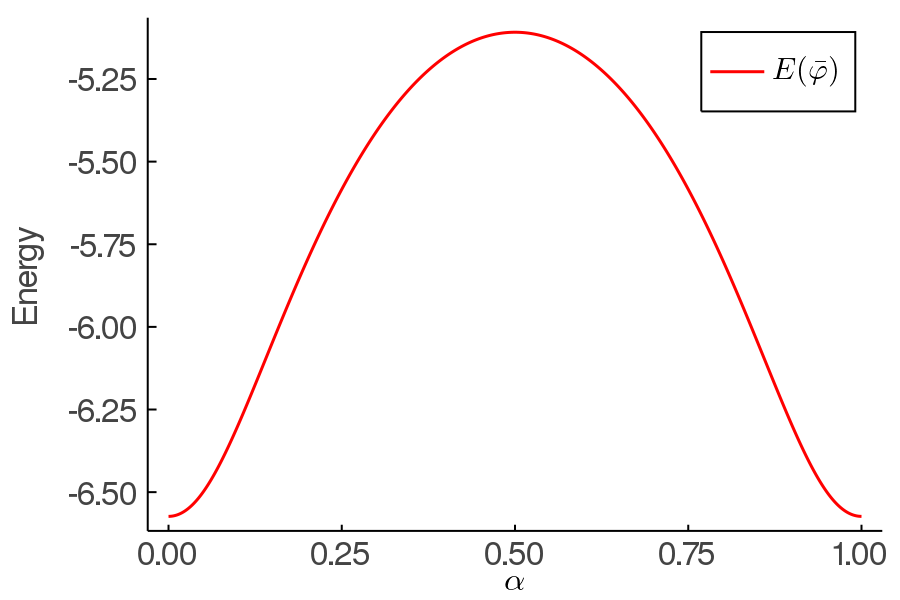}
\caption{(Example 3) Energy path of MEP}
\label{fig:ex3:Emep}
\end{minipage}
\hskip 0.5cm
\begin{minipage}[t]{0.48\textwidth}
\centering
\includegraphics[width=8.0cm]{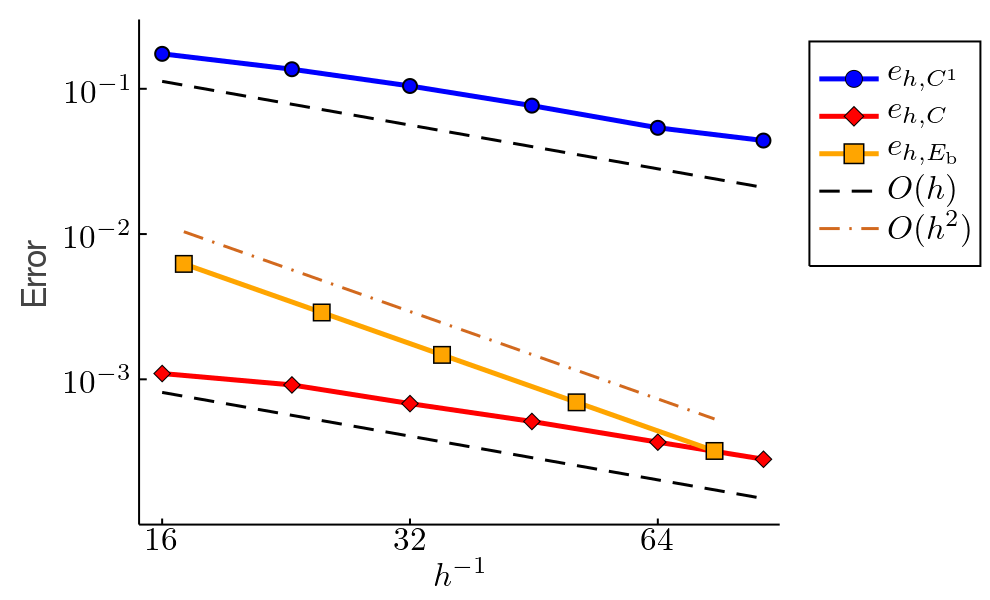}
\caption{(Example 3) Convergence of the discrete MEP.
}
\label{fig:ex3:convergence}
\end{minipage}
\end{figure}

\vskip 0.2cm

{\bf Example 3.}
Consider a two dimensional periodic system with two atoms lying in each unit cell. 
In the ``first" cell containing the origin, we fix one atom at the origin and put the other at the position $(x,y)\in \R^2$, and then construct a corresponding lattice $\Lambda := 
\left(\begin{smallmatrix}
2x & x \\
0 & y
\end{smallmatrix}\right)
\Z^2$ (see Figure \ref{fig:ex3:cell}). 
The interaction between any two atoms in the lattice $\Lambda$ is modelled by the Lennard-Jones pair potential
\begin{equation*}
    \phi(r) = 4\epsilon_0 \left( \left(\frac{\sigma_0}{r}\right)^{12} - \left(\frac{\sigma_0}{r}\right)^{6}\right)\phi_{\rm cut}(r)
\end{equation*}
with parameters $\epsilon_0=1,~\sigma_0 = 1$ and a cubic spline cut-off function \cite{Braun2020SharpUC} $\phi_{\rm cut}$ satisfying $\phi_{\rm cut}(r) = 1$ for $r\in[0,1.9]$ and $\phi_{\rm cut}(r) = 0$ for $r\geq2.7$.
Then the ``averaged" energy of the system (i.e. energy per cell) is the sum of the site energy of the two atoms located at origin and $(x,y)$,
\begin{equation*}
E(x,y) = \frac{1}{2}\sum_{(x',y')\in\Lambda}\phi\big(|(x',y')-(0,0)|\big) + \frac{1}{2}\sum_{(x',y')\in\Lambda}\phi\big(|(x',y')-(x,y)|\big).
\end{equation*}
We present the contour lines of the energy landscape in Figure \ref{fig:ex3:contour}, together with the two minimizers and the MEP.
We further present the energy path of the MEP in Figure \ref{fig:ex3:Emep}.
The decay of the numerical errors are shown in Figure \ref{fig:ex3:convergence}, from which we again observe an $\mathcal{O}(h)$ convergence rate for the discrete $C^1$-norm and discrete $C$-norm errors, and an $\mathcal{O}(h^2)$ convergence rate for the energy barrier error.

\section{Conclusions}
\label{sec:conclusion}
\setcounter{equation}{0}

This paper studies the convergence of the discrete MEPs that are given by the stationary states of the NEB methods. 
We establish an optimal convergence rate of the discrete MEP with respect to the number of images along the path. Our numerical tests demonstrate that the technical conditions we employed are indeed required. 

For numerical methods that find the MEP by evolving the path between local minimizers, say the NEB methods and string methods, an important remaining issue is the convergence with respect to the evolution time, which will be studied in our other work.

\section{Proofs: Convergence of the discrete MEP}
\label{sec:dmep}
\setcounter{equation}{0}
\setcounter{remark}{0}

In this section, we will provide a proof for the main result Theorem \ref{thm2:conv}. 
We will first propose reformulation of the MEP equation \eqref{mep}, for which we have proven linearized stability in \cite{liu2022stability}. 
We then formulate a corresponding discretization that inherits the stability of the reformulated continuous MEP equation. We will then use the fact that solutions to our original discrete MEP formulation~\eqref{dMEP:NEB} and the reformulated discrete MEP formulation coincide to establish the convergence of the discrete MEP. 
We emphasize that the linearized stability result we will prove is valid only for the reformulated discrete MEP operator and not the original discrete MEP operator~\eqref{dMEP:NEB}.

\subsection{The abstract MEP and its discretization}
\label{sec:defDMEP}

For convenience of the analysis, we shall reformulate the MEP equation \eqref{mep} as follows: 
Define an MEP operator $\mathscr{F} : \U \rightarrow \Cz\big([0,1];\R^N\big)$ with 
\begin{align}
\label{def:F}
\hskip -0.3cm \mathscr{F}(\varphi) (\alpha) :=
P^\perp_{\varphi'(\alpha)} \nabla E\big(\varphi(\alpha)\big) - \alpha(\alpha-1)(\alpha-\sad) \big(|\varphi'(\alpha)| - L(\varphi)\big) \frac{\varphi'(\alpha)}{\lvert \varphi'(\alpha) \rvert }
\end{align}
for $\alpha\in[0,1]$. 
The first term in \eqref{def:F} is perpendicular to the tangent direction $\varphi'$, which is exactly the same as the left-hand side of \eqref{mep1}; the second term in \eqref{def:F} is parallel to
the tangent $\varphi'$, which is designed to enforce the curve to be parameterized by normalized
arc length.
Then \eqref{mep} can rewritten in the following form: Find $\mv\in\U$ such that
\begin{equation}
\label{pb:mep}
\mathscr F(\mv) = 0.
\end{equation}

To give the first variation of $\mathscr{F}$ at $\mv$, we introduce a function measuring the gradient of MEP, which will be heavily used in our analysis. Let
\begin{equation}
\label{eq:ml}
    \ml(\alpha) := \nabla E\big(\mv(\alpha)\big)^{\rm T} \frac{\mv'(\alpha)}{\lvert \mv'(\alpha) \rvert^2},   \qquad \alpha\in[0,1] .
\end{equation}
We see from \eqref{mep1} that
\begin{equation}
\label{eq:ml1}
\nabla E\big(\mv(\alpha)\big) = \ml(\alpha) \mv'(\alpha)\qquad{\rm for}~\alpha\in[0,1],
\end{equation}
The following lemma \cite[Lemma 4.1]{MEPStab} states the properties of $\ml$, from which we can observe that $\sA=\ml'(0)$ and $\sB=\ml'(1)$.

\begin{lemma}
\label{lemma:lambda}
Let $\mv\in C^2\big( [0,1];\R^{N} \big)$ be the solution of \eqref{mep}.
If {\bf (A)} is satisfied, then
\begin{itemize}
\item[(i)]
$\big(\ml'(0), \mv'(0)\big)$, $\big(\ml'(\sad), \mv'(\sad)\big)$ and $\big(\ml'(1), \mv'(1)\big)$ are eigenpairs of the Hessians $\nabla^2 E(y^A_{\rm M})$, $\nabla^2 E(y_{\rm S})$ and $\nabla^2 E(y^B_{\rm M})$, respectively;
\item[(ii)]
$\ml'(0)>0,~\ml'(1)>0$ and $\ml'(\sad)<0$; and 
\item[(iii)]
there exist positive constants $\underline{c}, \bar{c}$ depending only on $\mv$, such that
\begin{equation}
\label{eq:mlequ}
    \underline{c} \leq \left\lvert \frac{\ml(\alpha)}{\alpha(\alpha-1)(\alpha-\sad)} \right\rvert \leq \bar{c} \qquad \forall~\alpha \in (0,\sad)\cup(\sad,1) .
\end{equation}
\end{itemize} 
\end{lemma}

Now we can derive the first variation of $\mathscr{F}$.
Let 
\begin{align*}
\X := \Big\{\tv\in\Co\big([0,1];\R^N\big):~\tv(0) = \tv(1) = 0 \Big\} 
\qquad {\rm and} \qquad
Y := C\big([0,1];\R^N\big) .
\end{align*}
By using \eqref{eq:ml1}, we obtain by a direct calculation that the first variation $\delta\mathscr{F}(\mv) : \X \rightarrow Y$ is
\begin{align}
\label{def:dF}
\delta\mathscr{F}(\mv)\tv (\alpha) &= 
P^\perp_{\mv'(\alpha)} \Big( \nabla^2 E\big(\mv(\alpha)\big)\tv(\alpha) - \ml(\alpha) \tv'(\alpha) \Big)
\nonumber \\[1ex]
&\hskip -1.8cm +  \alpha(\alpha-1)(\alpha-\sad)\left( \frac{\mv'(s)^{\rm T}\tv'(s)}{|\mv'(s)|} - \int_0^1\frac{\mv'(s)^{\rm T}\tv'(s)}{|\mv'(s)|}\dd s \right) \frac{\mv'(\alpha)}{\lvert \mv'(\alpha) \rvert}\qquad{\rm for}~\alpha\in[0,1].
\end{align}

\vskip 0.2cm

Next, we provide a discretization of the operator $\mathscr{F}$, with $M+1$ ($M\in\Z_+$) images.
Let $h=\frac{1}{M}$ be the corresponding mesh size.
We denote by $\alpha_k := k h ~(k=0,\cdots,M)$ the nodes of the discrete path.
Since there is a saddle point $y_{\rm S}$ along the MEP $\mv$ that reaches the maximum of $E$, there exists a $\ds\in\Z_+$ such that $\alpha_{\ds} \in (\sad-h,\sad+h)$ approximates the saddle $\sad$ and
\begin{align}
\label{sbar}
E\big(\mv(\alpha_{\ds})\big) = \max_{0\leq k\leq M} E\big(\mv(\alpha_k)\big) .
\end{align}
To approximate the tangent of a discrete path, we next introduce a first-order difference operator $\tilde{D}_h :(\R^N)^{M+1}\rightarrow (\R^N)^{M+1}$ with
\begin{align}
\label{def:Dh} 
(\tilde{D}_h\dv_h)_k  &:=
\left\{\begin{array}{ll}
\displaystyle 
\frac{\varphi_{h,k+1} - \varphi_{h,k}}{h}   &{\rm for}~k=0,\cdots,\ds-1  
\\[1ex]
\displaystyle 
\frac{\varphi_{h,k+1} - \varphi_{h,k-1}}{2h} 
& {\rm for}~ k=\ds
\\[1ex]
\displaystyle 
\frac{\varphi_{h,k} - \varphi_{h,k-1}}{h} 
& {\rm for}~ k=\ds+1,\cdots,M
\end{array} \right. .
\end{align}
We see from \eqref{sbar} that $\tilde{D}_h$ is equivalent to $D_h$ defined in \eqref{eq:upwind} when $\varphi_h$ is the uniform sampling nodes of $\mv$.
We shall also need the following difference operator $D^{-}_h :(\R^N)^{M+1}\rightarrow (\R^N)^{M+1}$ for arc length normalization, with 
\begin{align}
\label{def:DhL}
(D^{-}_h \varphi_h)_k &:=
\left\{\begin{array}{ll}
\displaystyle 
\frac{\varphi_{h,k+1} - \varphi_{h,k}}{h}
&{\rm for}~k=0
\\[1ex]
\displaystyle 
\frac{\varphi_{h,k} - \varphi_{h,k-1}}{h} 
& {\rm for}~ k=1,\cdots,M
\end{array} \right. .
\end{align}

Now we are able to provide a discretization {$\mathscr{F}_h$ of $\mathscr{F}$.} 
Let 
\begin{align*}
\X_{h} := \Big\{ \tv_h = \{\tv_{h,k}\}_{k=0}^{M}\in (\R^N)^{M+1}:~\tv_{h,0}=\tv_{h,M}=0 \Big\}
\end{align*}
equipped with the norm 
$$
\| \tv_h \|_{\X_h} :=|D^{-}_h \tv_h|_{\infty} + |\tv_h|_{\infty} ,
$$ 
and
\begin{equation*}
Y_h := \Big\{f_h = \{f_{h,k}\}_{k=0}^{M} \in (\R^N)^{M+1}:~f_{h,0}=f_{h,M}=0 \Big\}
\end{equation*}
equipped with the norm 
\begin{equation}
\label{def:Clh}
\| f_h \|_{\Y_h} := \max_{1\leq k\leq M-1}\left|\frac{f_{h,k}}{\alpha_k(\alpha_k-1)}\right|_{\infty} + \max_{ \substack{0\leq k\leq M, \\[0.05cm] k\neq\ds}}\left|\frac{f_{h,k}-f_{h,\ds}}{\alpha_k-\alpha_{\ds}}\right|_{\infty}.
\end{equation} 
We define $\mathscr{F}_h : \U_h \rightarrow Y_h$ by 
\begin{multline}
\label{def:Fh}
\big( \mathscr{F}_h(\varphi_h) \big)_k 
:= P_{(\tilde{D}_h\varphi_h)_k}^\perp \nabla E(\varphi_{h,k})
\\[1ex]
- \alpha_k(\alpha_k-1)(\alpha_k - \alpha_{\ds+\frac12}) \Big(|(D^{-}_h \varphi_h)_k| - L_h(\varphi_h)\Big) \frac{(\tilde{D}_h\varphi_h)_k}{|(\tilde{D}_h\varphi_h)_k|}
\qquad {\rm for}~k=0,\cdots,M ,
\quad
\end{multline}
where $\alpha_{\ds+\frac12} := (\ds+\frac{1}{2})h$ and $L_h:\U_h\rightarrow \R$ is a discretization of the length operator $L$
\begin{align}
\label{def:Lh}
L_h(\varphi_h) := h \sum_{k=1}^{M} |(D^{-}_h \varphi_h)_k| = \sum_{k=1}^{M} |\varphi_{h,k}-\varphi_{h,k-1}|.
\end{align}
We can then provide a discretization of the equation \eqref{pb:mep}: 
Find $\mv_h\in \U_h$ such that
\begin{equation}
\label{pb:dmep}
\mathscr{F}_h (\mv_h) = 0.
\end{equation}
We observe from \eqref{def:DhL}, \eqref{def:Lh} and \eqref{pb:dmep} that the images on the solution $\mv_h$ are uniformly distributed:
\begin{eqnarray}
\label{eq:eqdis}
\lvert \mv_{h,k} - \mv_{h,k-1} \rvert = \lvert \mv_{h,k} - \mv_{h,k+1}\rvert
\qquad\forall~ k=1,\cdots,M-1.
\end{eqnarray} 
The solution $\mv_h$ of \eqref{pb:dmep} can be viewed as a discrete MEP within the abstract framework in this section.
To bridge the gap between the discrete MEPs from \eqref{pb:dmep} and \eqref{dMEP:NEB}, we will show in the last part of this proof (in Section \ref{sec:conv}) that the solution of the abstract form \eqref{pb:dmep} is equivalent to that of \eqref{dMEP:NEB}.

For $\varphi_h \in \U_h$, we can derive the first variation $\delta\mathscr{F}_h(\varphi_h) : \X_h \rightarrow \Y_h$ with
\begin{align}
\label{def:dFh}
& \big( \delta\mathscr{F}_h(\varphi_h) \tdmv \big)_k 
\nonumber 
\\[1ex]
=~ & P_{(\tilde{D}_h\varphi_h)_k}^\perp \left( \nabla^2 E(\varphi_{h,k})\tv_{h,k} - \frac{ \nabla E(\varphi_{h,k})^{\rm T} (\tilde{D}_h\varphi_h)_k}{|(\tilde{D}_h\varphi_h)_k|^2} (\tilde{D}_h\tv_h)_k\right) 
\nonumber 
\\[1ex]
& - \frac{ \big( P_{(\tilde{D}_h\varphi_h)_k}^\perp \nabla E(\varphi_{h,k}) \big)^{\rm T} (\tilde{D}_h\tv_h)_k}{|(\tilde{D}_h\varphi_h)_k|} \frac{(\tilde{D}_h\varphi_h)_k}{|(\tilde{D}_h\varphi_h)_k|}
\nonumber 
\\[1ex]
& + \alpha_k(\alpha_k-1)(\alpha_k - \alpha_{\ds+\frac12}) \left( \frac{ (D^{-}_h \tv_h)_k^{\rm T} (D^{-}_h\varphi_h)_k}{|(D^{-}_h\varphi_h)_k|} - h \sum_{t=1}^{M}\frac{ (D^{-}_h\tv_h)_t^{\rm T} (D^{-}_h\varphi_h)_t}{|(D^{-}_h \varphi_h)_t|} \right) \frac{(\tilde{D}_h\varphi_h)_k}{|(\tilde{D}_h\varphi_h)_k|}
\nonumber
\\[1ex]
& + \alpha_k(\alpha_k-1)(\alpha_k - \alpha_{\ds+\frac12})\big(|(D^{-}_h\varphi_h)_k| - L_h(\varphi_h)\big)  \frac{ P_{(\tilde{D}_h\varphi_h)_k}^\perp (\tilde{D}_h\tv_h)_k}{|(\tilde{D}_h\varphi_h)_k|} 
\end{align}
for $k=0,\cdots,M$.

A key ingredient of our convergence analysis is to show the stability of the discrete MEP, which gives the bound of the inverse of $\delta\mathscr{F}_h(\Pi_h\mv)$ and will be presented in the coming subsection. 
Note that we have justified the stability of a continuous MEP in our previous work \cite{MEPStab}, but generalizing the stability result to the discrete path requires non-trivial changes to those arguments.

\subsection{Stability of the discrete MEP}
\label{sec:stab_dmep}

Define a projection operator $\Pi_h : \Cz\big([0,1];\R^N\big) \rightarrow (\R^N)^{M+1}$ by
\begin{equation}
\label{Pih}
( \Pi_h\varphi )_k := \varphi(k h), \qquad k=0,\cdots,M.
\end{equation}
The following theorem states the stability of the discrete MEP. 

\begin{theorem}
\label{le:stab_dmep}
Let $\mv\in C^3\big( [0,1];\R^{N} \big)$ be a solution of \eqref{mep}.
Assume that {\bf (A)} and {\bf (B)} are satisfied. 
Then for sufficiently small $h$, there exists a constant $C>0$ depending only on $E$ and $\mv$ such that
\begin{equation}
\label{lemme:stability}
\| \tdmv \|_{\X_h} \leq C \|\delta\mathscr{F}_h(\Pi_h \mv)\tdmv \|_{\Y_h}  
\qquad\forall~\tdmv \in \X_h.
\end{equation}
\end{theorem}

We shall make some preparations for the proof this theorem.
Let 
\begin{align}
\hat{X}_h &:= \Big\{ \hat{\tv}_h \in \X : ~  \hat{\tv}_h|_{[\alpha_{k-1},\alpha_k]} \text{ is a cubic polynomial for} ~1\leq k \leq M 
\nonumber
\\[1ex]
\label{def:hatXh}
& \hskip 2cm \quad{\rm and}~ \hat{\tv}_h ~{\rm satisfying} ~~ \hat{\tv}'_h(\alpha_k) =  (\tilde{D}_h\Pi_h\hat{\tv}_h)_k \quad {\rm for}~ 0\leq k\leq M \Big\} 
\end{align}
equipped with the norm $\|\cdot\|_{\X}$.
We see that dim$(\hat{X}_h) = N(M-1)$ and any $\hat{\tv}_h\in\hat{X}_h$ can be determined by their values at the nodes $\alpha_k~(1\leq k\leq M-1)$. Thus $\Pi_h:\hat{X}_h\rightarrow \X_h$ is an isomorphism. Denote its inverse operator by $P_h : \X_h \rightarrow \hat{\X}_h$ satisfying
\begin{equation}
\label{def:Ph}
P_h \tv_h (\alpha_k) = \tv_{h,k} \qquad{\rm for}~k=0,\cdots,M.
\end{equation}
For a given $\hat{\tv}_h\in\hat{\X}_h$, let
\begin{equation}
\label{pb:tdmv}
f_h = \Pi_h \delta\mathscr{F}(\mv) \hat{\tv}_h.
\end{equation}
To prove Theorem \ref{le:stab_dmep}, we shall first bring in a basis set of $\hat{X}_h$ to derive a spectral representation of \eqref{pb:tdmv} and obtain some linear systems;
then by analyzing the coefficient matrix we will show that the operator $\Pi_h\delta\mathscr{F}(\mv):\hat{\X}_h\rightarrow \Y_h$ is invertible and $\big\| \big( \Pi_h\delta\mathscr{F}(\mv) \big)^{-1}\big\|_{\LL(\Y_h,\hat{\X}_h)}$ is uniformly bounded with respect to $h$;
and finally, Theorem \ref{le:stab_dmep} follows from a perturbation argument.

For $\alpha\in [0,1]$, we consider the eigenvalue problem 
\begin{equation}
\label{eq:ev}
\Big(P^\perp_{\mv'(\alpha)}\nabla^2 E\big(\mv(\alpha)\big) P^\perp_{\mv'(\alpha)}\Big) \evt_i(\alpha) = \ev_i(\alpha) \evt_i(\alpha)
\qquad {\rm for}~ i = 0,1,\cdots,N-1 
\end{equation}
with $\{z_i(\alpha)\}_{i=0}^{N-1}$ the eigenvalues and $\big\{\evt_i(\alpha)\big\}_{i=0}^{N-1}$ the corresponding eigenfunctions.
From \cite{MEPStab}, we can order the functions $z_i(\alpha)$ such that $\ev_0(\sad) \leq \ev_1(\sad) \leq \cdots \leq \ev_{N -1}(\sad)$ and $z_i\in C^2([0,1];\R),~\evt_i\in C^2\left([0,1];\R^{N}\right)$ for $i=0,1,\cdots, N-1$.
Moreover, we have $z_0(\alpha)\equiv 0$, $\evt_0(\alpha) \equiv \mv'(\alpha)/\lvert\mv'(\alpha)\rvert$ and 
\begin{align}\label{eq:ev-1}
 \ev_j(\alpha) > 0 \qquad{\rm for}~ j=1,\cdots,N-1,~\alpha=0,\sad,1 .
\end{align}
Let $g_{h,k,j}\in \hat{X}_h~ (1\leq k\leq M-1,~0\leq j\leq N-1)$ be a set of functions satisfying
\begin{align}
\label{eq:bas}
g_{h,k,j}(\alpha_t) = \delta_{kt}~\evt_j(\alpha_t)
\qquad {\rm for}~t=1,\cdots,M-1.
\end{align}
Note that $\big\{g_{h,k,j}\big\}$ gives a complete basis set for the finite dimensional space $\hat{X}_h$, since any function $\hat{\tv}_h\in\hat{X}_h$ can be determined by their values at the nodes $\alpha_k~(1\leq k\leq M-1)$.

Now we can represent $\hat{\tv}_h$ in \eqref{pb:tdmv} by the basis $\{ g_{h,k,j} \}$ ($1\leq k\leq M-1,~0\leq j\leq N-1$) 
\begin{equation}
\label{eq:tdmv}
\hat{\tv}_h(\alpha) = \sum_{k=1}^{M-1}\sum_{j=0}^{N-1}\drp_{k,j} g_{h,k,j}(\alpha)
\end{equation}
with $\{\drp_{k,j}\}$ the unknown coefficients. 
As the operator $\delta\mathscr{F}(\mv)$ \eqref{def:dF} has completely different behavior in the direction $\mv'$ and in the subspace $\mv'^\perp$, we will project \eqref{pb:tdmv} into these two subspaces respectively.
Define $\drp_{M,j}= \drp_{0,j} = 0$ for $j=0,\cdots,N-1$.
Denote $\drp^\perp \in (\R^{N-1})^{M+1} $ and $ \drp_0 \in \R^{M+1}$ with
\begin{equation}
\label{def:drpperp}
\drp^\perp_k  := \{ \drp_{k,j}\}_{j=1}^{N-1},\qquad (\drp_0)_k := \drp_{k,0} \qquad {\rm for}~k=0,\cdots,M,
\end{equation}
Denote $F^\perp\in(\R^{N-1})^{M-1}$ and $F_{0}\in\R^{M-1}$ with
\begin{equation}
\label{def:Fperp}
F^\perp_k := \left\{ \evt_j(\alpha_k)^{\rm T} f_{h,k} \right\}_{j=1}^{N-1},\qquad (F_{0})_k:=\evt_0(\alpha_k)^{\rm T} f_{h,k} \qquad{\rm for}~k=1,\cdots,M-1.
\end{equation}
We first consider the problem \eqref{pb:tdmv} in the subspace perpendicular to the tangent of MEP $\xi_0=\mv'/\lvert\mv'\rvert$.
By substituting \eqref{eq:tdmv} into \eqref{pb:tdmv}, taking $k$-th component, multiplying both sides with $\evt_j(\alpha_k)^{\rm T},~1\leq j\leq N-1$, using \eqref{eq:ml1},~\eqref{eq:ev} and \eqref{eq:bas}, we obtain by a direct calculation that
\begin{align}
\label{eq:drp_perp}
A(\alpha_k) \drp^\perp_{k} - \ml(\alpha_k) (\tilde{D}_h\drp^\perp)_k &= F^\perp_k \qquad{\rm for}~k=1,\cdots,M-1,
\end{align}
where 
\begin{equation}
\label{def:A1}
A(\alpha) := {\rm Diag}(\ev_1(\alpha),\ev_2(\alpha), \cdots,\ev_{N-1}(\alpha)) = 
\begin{pmatrix}
z_1(\alpha) & & &  \\[1ex]
 &  z_2(\alpha) &  &\\[1ex]
 & & \ddots &\\[1ex]
 & &  & z_{N-1}(\alpha)
\end{pmatrix}
\end{equation}
for $\alpha\in[0,1]$.
Note that the linear system \eqref{eq:drp_perp} is independent of $\drp_0$.

To better estimate $\|\hat{\tv}_h\|_{\X}$ by $\|f_h\|_{\Y_h}$, we transform the linear system \eqref{eq:drp_perp} into another form.
We know from \eqref{def:Dh} that $(\tilde{D}_h \drp^\perp)_{\ds} = \frac12\big((\tilde{D}_h \drp^\perp)_{\ds+1} + (\tilde{D}_h \drp^\perp)_{\ds-1}\big)$. Taking $k=\ds-1,~\ds,~\ds+1$ in \eqref{eq:drp_perp}, we obtain the following linear system:
\begin{equation}\label{eq:drp_perp1}
P_{h,\sad}
\begin{pmatrix}
(\tilde{D}_h\drp^\perp)_{\ds-1} \\[1ex]
(\tilde{D}_h\drp^\perp)_{\ds+1} \\[1ex]
\drp^\perp_{\ds}
\end{pmatrix} = 
\begin{pmatrix}
\frac{ F^\perp_{\ds-1} - F^\perp_{\ds} }{-h} \\[1ex]
\frac{F^\perp_{\ds+1} - F^\perp_{\ds} }{h} \\[1ex]
F^\perp_{\ds}
\end{pmatrix},
\end{equation}
where 
\begin{align}
\label{def:Ph-sad}
& P_{h,\sad}:=
\begin{pmatrix}
B_{\ds-1} -\frac{\ml(\alpha_{\ds})}{2h} I & -\frac{\ml(\alpha_{\ds})}{2h} I & \frac{A(\alpha_{\ds}) - A(\alpha_{\ds-1})}{h}   \\[1ex]
-\frac{\ml(\alpha_{\ds})}{2h} I & B_{\ds+1} -\frac{\ml(\alpha_{\ds})}{2h} I& \frac{A(\alpha_{\ds+1}) - A(\alpha_{\ds})}{h} \\[1ex]
-\frac{\ml(\alpha_{\ds})}{2} I & -\frac{\ml(\alpha_{\ds})}{2} I & A(\alpha_{\ds})
\end{pmatrix} \qquad {\rm with}
\\[1ex] \nonumber
& B_{\ds-1} := I + h \frac{A(\alpha_{\ds-1})}{\ml(\alpha_{\ds-1})} \qquad{\rm and}\qquad
 B_{\ds+1} := I - h \frac{A(\alpha_{\ds+1})}{\ml(\alpha_{\ds+1})}.
\end{align}
For any $\zeta\in(0,\frac{\sad}{2})$, there exists a $\ddelta \in \Z_{+}$ such that 
\begin{equation}
\label{eq:ddelta}
\ddelta h \leq \zeta < (\ddelta + 1)h.
\end{equation}
Taking $k=1,\cdots,\ds-2$ in \eqref{eq:drp_perp} and combining with
\begin{equation*}
-h\sum_{t=k}^{\ds-1} (\tilde{D}_h\drp^\perp)_t + \drp^\perp_{\ds} = \drp^\perp_k =  h\sum_{t=0}^{k-1} (\tilde{D}_h\drp^\perp)_t,\quad h \sum_{t=0}^{\ddelta} (\tilde{D}_h\drp^\perp)_t = \drp^\perp_{\ddelta+1}, 
\end{equation*}
we obtain the following linear systems:
\begin{align}
\label{eq:drp_perp2}
& \frac{B_k}{\alpha_{k}-\alpha_{\ds}} (\tilde{D}_h\drp^\perp)_{k} 
= \frac{A(\alpha_k) - A(\alpha_{\ds})}{\alpha_{k}-\alpha_{\ds}} \drp^\perp_{\ds} + \frac{\ml(\alpha_{\ds})}{\alpha_{k}-\alpha_{\ds}} (\tilde{D}_h\drp^\perp)_{\ds} 
\nonumber
\\[1ex]
& \hskip 3cm - \frac{F^\perp_k - F^\perp_{\ds}}{\alpha_{k}-\alpha_{\ds}}  - \frac{h A(\alpha_k)}{\alpha_{k}-\alpha_{\ds}}\sum_{t=k+1}^{\ds-1} (\tilde{D}_h\drp^\perp)_t 
\qquad{\rm for}~\ds-\ddelta\leq k\leq \ds-2,
\\[1ex]
\label{eq:drp_perp3}
& \hskip 1.2cm  B_k  (\tilde{D}_h\drp^\perp)_{k}  = - h A(\alpha_k)\sum_{t=k+1}^{\ds-1} (\tilde{D}_h\drp^\perp)_t + A(\alpha_k) \drp^\perp_{\ds} - F^\perp_k 
\end{align}
for $\ddelta+1\leq k\leq \ds-\ddelta-1$ and
\begin{equation}
\label{eq:drp_perp4}
    P_{h,\zeta} \begin{pmatrix}
    (\tilde{D}_h\drp^\perp)_0 \\[1ex]
    (\tilde{D}_h\drp^\perp)_1 \\[1ex]
    \vdots \\[1ex]
    (\tilde{D}_h\drp^\perp)_{\ddelta}
    \end{pmatrix} = \begin{pmatrix}
    \drp^\perp_{\ddelta+1} \\[1ex]
    \frac{F^\perp_1}{\alpha_1} \\[1ex]
    \cdots \\[1ex]
    \frac{F^\perp_{\ddelta}}{\alpha_{\ddelta}}
    \end{pmatrix},
\end{equation}
where 
\begin{align}
\label{eq:Bk}
B_k &:= I + h \frac{A(\alpha_{k})}{\ml(\alpha_{k})}  \qquad{\rm for}~ 1\leq k \leq \ds-2 \hskip 2cm {\rm and}
\\[1ex]
\label{def:Phd}
P_{h,\zeta} &:= \begin{pmatrix}
h I & h I & h I &  \cdots & h I \\[1ex]
\frac{h A(\alpha_1)}{\alpha_1} & -\frac{\ml(\alpha_1)}{\alpha_1}I &   \\[1ex]
\frac{h A(\alpha_2)}{\alpha_2} & \frac{h A(\alpha_2)}{\alpha_2} & -\frac{\ml(\alpha_2)}{\alpha_2}I & & &  
\\[1ex]
\vdots & \vdots & \vdots & \ddots &  & \\[1ex]
\frac{h A(\alpha_{\ddelta})}{\alpha_{\ddelta}} & \frac{h A(\alpha_{\ddelta})}{\alpha_{\ddelta}} & \frac{h A(\alpha_{\ddelta})}{\alpha_{\ddelta}} & \cdots & -\frac{\ml(\alpha_{\ddelta})}{\alpha_{\ddelta}}I
\end{pmatrix}.
\end{align}

We have rewritten \eqref{eq:drp_perp} into \eqref{eq:drp_perp1}, \eqref{eq:drp_perp2}, \eqref{eq:drp_perp3} and \eqref{eq:drp_perp4}, which are used to bound $(\tilde{D}_h\drp^\perp)_k$ at the saddle, near the saddle, between the saddle and the minimizer $y^A_{\rm M}$, and near the minimizer $y^A_{\rm M}$, respectively. For $k>\ds+1$ we have the similar linear systems. 
To estimate $|\tilde{D}_h\drp^\perp|_{\infty}$, we only need to analyze the coefficient matrix in these linear systems.

We then consider the problem in the tangent direction $\mv'$.
By substituting \eqref{eq:tdmv} into \eqref{pb:tdmv}, taking $k$-th component, multiplying both sides with $\evt_0(\alpha_k)^{\rm T}$, we obtain
\begin{align}
\label{eq:drp0}
& \alpha_k(\alpha_k-1)(\alpha_k-\sad)\left( (\tilde{D}_h\drp_0)_k - \sum_{i=0}^{N-1}\sum_{t=1}^{M-1}\drp_{t,i}  c_{h,t,i} \right) = (F_{0})_k\qquad{\rm for}~k=1,\cdots,M-1,
\end{align}
where
\begin{align}
\label{def:c_hti}
c_{h,t,i} & :=  \int_0^1 g_{h,t,i}(s)^{\rm T}\evt'_0(s)\dd s\quad{\rm for}~\alpha\in[0,1].
\end{align}
We see that $\drp_0$ depends on the component $\drp^\perp$.
In our analysis, we will first focus on the problem \eqref{eq:drp_perp} in the subspace $\mv'^\perp$, and then put things back into \eqref{eq:drp0} in tangent direction $\mv'$. 

So far, we have obtained a spectral representation of \eqref{pb:tdmv}. 
We shall next choose an appropriate $\zeta$ and turn to the analysis of the coefficient matrices in \eqref{eq:drp_perp1}, \eqref{eq:drp_perp2}, \eqref{eq:drp_perp3} and \eqref{eq:drp_perp4}.
Define 
\begin{align}
\label{def:sigmaj}
& \sigma_j :=  \frac{\ev_j(0)}{\ml'(0)} \qquad{\rm for}~1\leq j\leq N-1 \qquad{\rm and}
\\[1ex]
\label{eq:n_j}
& n_j(\alpha) := \frac{\ev_j(\alpha)}{\ml(\alpha)} - \frac{\sigma_j}{\alpha}
\quad\quad{\rm for}~\alpha\in \big(0,\frac{\sad}{2}\big],~1\leq j\leq N-1.
\end{align}
We have from assumption {\bf (B)} that $\sigma_j > 1$ for $1\leq j\leq N-1$. Since
\begin{align}
\label{estimate:n}
\left\lvert \lim_{\alpha\rightarrow 0^+} n_j(\alpha)\right\rvert &= \left\lvert \lim_{\alpha\rightarrow 1^-}\frac{z_j(\alpha)\alpha - \sigma_j\ml(\alpha)}{\alpha\ml(\alpha)}\right\rvert = \left\lvert \lim_{\alpha\rightarrow 1^-}\frac{z'_j(\alpha)\alpha + z_j(\alpha) - \sigma_j\ml'(\alpha)}{\alpha\ml'(\alpha) + \ml(\alpha)} \right\rvert
\nonumber
\\[1ex]
& = \left\lvert \lim_{\alpha\rightarrow 0^+}\frac{2z'_j(\alpha) + z''_j(\alpha)\alpha - \sigma_j\ml''(\alpha)}{\alpha\ml''(\alpha) + 2\ml'(\alpha)} \right\rvert = \left\lvert \frac{2z'_j(0)- \sigma_j\ml''(0)}{2\ml'(0)} \right\rvert  < C
\end{align}
with $C$ a constant depending only on $\ml$ and $z_j$, we have that $n_j$ is bounded on $(0,\frac{\sad}{2}]$. 
Define
\begin{equation}
\label{def:bar-sigma}
\bar{\sigma} := \frac{1}{2} \left( 1 + \max_{1\leq j\leq N-1} \sigma_j \right).
\end{equation}
From the fact that $\sigma_j > 1$ for $1\leq j\leq N-1$ , we have $\bar{\sigma} > 1$.
Therefore, there exists $\zeta\in(0,\frac{\sad}{2})$ such that $\ev_j(\alpha)>0$ for $\alpha\in[\sad-\zeta-2h,\sad+h],~1\leq j\leq N-1$, $\frac{1}{2}\ml'(0) \leq \lVert \ml'\rVert_{\Cz([0,\zeta];\R)} \leq 2 \ml'(0)$ and
\begin{equation}
\label{eq:delta}
\frac{\sigma_j -\bar{\sigma}}{\zeta} \geq \|n_j\|_{\Cz((0,\frac{\sad}{2}];\R)} \qquad {\rm for}~1\leq j\leq N-1. 
\end{equation}
This together with \eqref{eq:n_j} implies
\begin{equation}
\label{eq:delta-d}
\frac{\ev_j(\alpha)}{\ml(\alpha)} = \frac{\sigma_j}{\alpha} + n_j(\alpha) \geq \frac{\bar{\sigma}}{\alpha}
\qquad\forall~\alpha\in(0,\zeta] .
\end{equation}

\begin{lemma}
\label{le:Bk}
Assume that the assumptions of Theorem \ref{le:stab_dmep} are satisfied.
Let $\bar\sigma$, $\zeta$, $\ddelta$ and $B_k$' be given by \eqref{def:bar-sigma}, \eqref{eq:delta}, \eqref{eq:ddelta} and \eqref{eq:Bk}, respectively. 
Then for sufficient small $h$, $B_k$'s are invertible and there exist a constant $C>0$ depending only on $\ml$ and $A$ such that 
\begin{align}
\label{eq:Bk1}
\lVert B_k^{-1} \rVert_{\infty} \leq \frac{\bar{\sigma}}{k+\bar{\sigma}} \qquad &{\rm for}~ 1\leq k\leq \ddelta,
\\[1ex]
\label{eq:Bk2}
\lVert B_k^{-1} \rVert_{\infty} \leq 1+Ch
\qquad &{\rm for}~ \ddelta+1 \leq k\leq \ds-\ddelta-1, 
\\[1ex]
\label{eq:Bk3}
\lVert B_k^{-1} \rVert_{\infty} < \frac{1}{ \alpha_{\ds}- \alpha_k }
\qquad &{\rm for}~ \ds-\ddelta\leq k\leq \ds-2,
\end{align}
where the constants $C$ depend only on $\ml$ and $A$. Moreover, we have
\begin{equation}
\label{eq:Bk4}
\lVert B_k^{-1} - I \rVert_{\infty} \leq \frac{C}{k} 
\qquad {\rm for}~ 1\leq k\leq \ddelta.
\end{equation}
\end{lemma}

\begin{proof}
Since $B_k$ is a diagonal matrix, we have
\begin{equation}
\label{eq:Bk-1}
\lVert B_k^{-1} \rVert_{\infty} 
= \max_{1\leq j\leq N-1} \left| \frac{1}{1+ \frac{z_j(\alpha_k) h}{\ml(\alpha_k)}}\right| \qquad{\rm and}\qquad \lVert B_k^{-1}-I \rVert_{\infty} 
= \max_{1\leq j\leq N-1} \left|\frac{ \frac{z_j(\alpha_k) h}{\ml(\alpha_k)} }{1+\frac{z_j(\alpha_k) h}{\ml(\alpha_k)}} \right|
\end{equation}
for $1\leq k\leq \ds-1$.
For $1\leq k\leq \ddelta$,  using \eqref{def:bar-sigma} and \eqref{eq:delta-d} yields
\begin{align*}
1 + \frac{\ev_j(\alpha_k)}{\ml(\alpha_k)}h \geq 1 + \frac{\bar{\sigma} h }{\alpha_k}\geq 1 + \frac{\bar{\sigma}}{k} \qquad{\rm for}~1\leq j\leq N-1.
\end{align*}
This together with \eqref{eq:Bk-1} leads to
\begin{align*}
\lVert B_k^{-1} \rVert_{\infty} \leq \frac{k}{k+\bar{\sigma}} \qquad{\rm and}\qquad
\lVert B_k^{-1}-I \rVert_{\infty}  \leq  \frac{C h}{\ml(\alpha_k)}\leq \frac{C}{k} \frac{\alpha_{k}}{\int_0^{\alpha_k} \lvert\ml'(s)\rvert \dd s} \leq \frac{C}{k} ,
\end{align*}
where $C$ depends on $\ml$ and $A$. 
This completes the proof for \eqref{eq:Bk1} and \eqref{eq:Bk4}.
	
For $\ddelta+1 \leq k\leq \ds-\ddelta-1$, since $\lvert \ml(\alpha_k) \rvert$ have both upper bound and lower bound, we have from \eqref{eq:Bk-1} that
\begin{align*}
\lVert B_k^{-1} \rVert_{\infty} \leq \frac{1}{1 - Ch} \leq 1 + Ch ,
\end{align*}
where $C$ depends on $\ml$ and $A$. This completes the proof for \eqref{eq:Bk2}.
	
For $\ds-\ddelta\leq k\leq \ds-2$, using the fact that $\ev_j(\alpha_k) > 0$ and $\ml(\alpha_k)>0$, we obtain
\begin{equation*}
\left\lVert B_k ^{-1} \right\rVert_{\infty} 
= \max_{1\leq j\leq N-1} \frac{1}{1+ \frac{z_j(\alpha_k) h}{\ml(\alpha_k)}} < 1 < \frac{1}{\alpha_{\ds}-\alpha_k}. \qedhere
\end{equation*}
\end{proof}

\begin{lemma}
\label{le:Ph-d}
Assume that the assumptions of Theorem \ref{le:stab_dmep} are satisfied.
Let $\zeta$, $\ddelta$, $P_{h,\sad}$ and $P_{h,\zeta}$ be given by \eqref{eq:delta}, \eqref{eq:ddelta}, \eqref{def:Ph-sad} and \eqref{def:Phd}, respectively.
Then for sufficiently small $h$, $P_{h,\sad}$ and $P_{h,\zeta}$ are invertible and there exist constants $C>0$ depending on $\ml$ and $A$ such that
\begin{equation}
\label{est:Ph}
\lVert P_{h,\sad}^{-1} \rVert_{\infty} \leq C
\qquad{\rm and}\qquad 
\lVert P_{h,\zeta}^{-1} \rVert_{\infty} \leq C.
\end{equation}
\end{lemma}

\begin{proof}
Since $\displaystyle \lim_{h\rightarrow 0^+}\alpha_{\ds} = \sad$ and $\ml(\sad)=0$, we have
\begin{align*}
    P_{\sad} := \lim_{h\rightarrow 0^+} P_{h,\sad} = 
    \begin{pmatrix}
    B_{\sad}  &  & A'(\sad)   \\[1ex]
    & B_{\sad} & A'(\sad) \\[1ex]
    &  & A(\sad)
    \end{pmatrix},
    \qquad P_{\sad}^{-1} = 
    \begin{pmatrix}
    B_{\sad}^{-1}  &  & - B_{\sad}^{-1} A'(\sad)A(\sad)^{-1}   \\[1ex]
    & B_{\sad}^{-1} & - B_{\sad}^{-1} A'(\sad)A(\sad)^{-1}  \\[1ex]
    &  & A(\sad)^{-1}
    \end{pmatrix},
\end{align*}
where $B_{\sad}:= I - \frac{A(\sad)}{\ml'(\sad)} \in \R^{(N-1)\times(N-1)}$. From Lemma \ref{lemma:lambda} (ii) and the fact that $\ev_j(\sad)>0$ for $1\leq j\leq N-1$, we obtain
\begin{equation*}
\lVert B_{\sad}^{-1} \rVert_{\infty} \leq \max_{1\leq j\leq N-1} \frac{1}{1-\frac{\ev_j(\sad)}{\ml'(\sad)}}\leq 1,
\qquad \lVert A(\sad)^{-1} \rVert_{\infty} \leq \max_{1\leq j\leq N-1} \frac{1}{\ev_j(\sad)}.
\end{equation*}
This implies 
\begin{equation*}
\lVert P_{\sad}^{-1} \rVert_{\infty} 
\leq \lVert B_{\sad}^{-1} \rVert_{\infty} + \lVert B_{\sad}^{-1} \rVert_{\infty}\lVert A'(\sad) \rVert_{\infty}\lVert A(\sad)^{-1} \rVert_{\infty} +\lVert A(\sad)^{-1} \rVert_{\infty} \leq C,
\end{equation*}
where $C$ depends only on $\ml$ and $A$. Thus $P_{h,\sad}$ is inveritible and $\lVert P_{h,\sad}\rVert_{\infty}\leq C$ for sufficiently small $h$. This completes the proof for first part of \eqref{est:Ph}.
    
We then investigate the inverse of $P_{h,\zeta}$:
\begin{equation}
\label{eq:Ph-d-1}
P_{h,\zeta}^{-1} = Q_{h,\zeta} T_{h,\zeta},
\end{equation}
where 
\begin{align*}
	Q_{h,\zeta} &:=
	\begin{pmatrix}
	h^{-1} B_{1}^{-1}\cdots B_{\ddelta}^{-1} & -B_{1}^{-1} & -B_1^{-1}B_2^{-1} & \cdots & -B_{1}^{-1}\cdots B_{\ddelta}^{-1} \\[1ex]
	h^{-1} (I-B_1^{-1})B_2^{-1}\cdots B_{\ddelta}^{-1} & B_1^{-1} & (B_1^{-1}-I)B_2^{-1} & \cdots & (B_1^{-1}-I)B_2^{-1}\cdots B_{\ddelta}^{-1} \\[1ex]
	h^{-1} (I-B_2^{-1})B_3^{-1}\cdots B_{\ddelta}^{-1} &  & B_2^{-1} & \cdots & (B_2^{-1}-I)B_3^{-1}\cdots B_{\ddelta}^{-1} \\[1ex]
	\vdots & & & \ddots & \\[1ex]
	h^{-1} (I-B_{\ddelta}^{-1}) & & & & B_{\ddelta}^{-1}
	\end{pmatrix}, \\[1ex]
	T_{h,\zeta} &:= \begin{pmatrix}
	I & & & \\[1ex]
	& -\frac{\alpha_1}{\ml(\alpha_1)} I & & \\[1ex]
	& & \ddots & \\[1ex]
	& & & -\frac{\alpha_{\ddelta}}{\ml(\alpha_{\ddelta})} I
	\end{pmatrix}.
\end{align*}
We see from Lemma \ref{lemma:lambda} (iii) that $\lVert T_{h,\zeta}\rVert_{\infty}\leq C$ where $C$ only depends on $\ml$ and $\zeta$.
To estimate $\lVert Q_{h,\zeta}\rVert_{\infty}$, we first present two inequalities. 
We claim that there exist constants $C_1,C_2,C>0$ depending only on $\bar{\sigma}$ such that
\begin{align}
\label{eq:ineq1}
C_1 \frac{1}{n^{\bar{\sigma}}} \leq \prod_{s=1}^{n} \frac{s}{s+\bar{\sigma}} \leq C_2 \frac{1}{n^{\bar{\sigma}}} \qquad &{\rm for}~ n \in \Z_+ \qquad{\rm and}
\\[1ex]
\label{eq:ineq2}
\frac{1}{n^{\bar{\sigma}}} \leq C \left(\frac{1}{(n-1)^{\bar{\sigma}-1}} - \frac{1}{n^{\bar{\sigma}-1}} \right) \qquad &{\rm for}~ n\geq 2,
\end{align}
The inequality \eqref{eq:ineq1} results from induction and the fact that
\begin{equation*}
\lim_{n\rightarrow +\infty} \frac{n+1}{n+1+\bar{\sigma}} = 1 = \lim_{n\rightarrow +\infty}\frac{n^{\bar\sigma}}{(n+1)^{\bar\sigma}} .
\end{equation*}
Using the monotonicity, there exists a constant $C>0$ depending on $\bar{\sigma}$ such that
\begin{equation*}
\frac{1}{n} \leq \frac{1}{2} \leq C (2^{\bar{\sigma}-1} - 1) \leq C \left(\frac{n^{\bar{\sigma}-1}}{(n-1)^{\bar{\sigma}-1}} - 1 \right) \qquad\forall~n\geq 2,
\end{equation*}
which leads to \eqref{eq:ineq2} by multiplying $n^{1-\bar{\sigma}}$ on both sides.
	
Applying \eqref{eq:ineq1} and Lemma \ref{le:Bk}, we first estimate
\begin{equation}
\label{eq:ineq3}
\left\lVert \prod_{t=k+1}^{k'} B_{t}^{-1} \right\rVert_{\infty} \leq  \prod_{t=k+1}^{k'} \frac{t}{t+\bar{\sigma}} 
\leq  C \frac{(k+1)^{\bar{\sigma}}}{(k')^{\bar{\sigma}}} \leq C(k+1)^{\bar{\sigma}} \left(\frac{1}{(k'-1)^{\bar{\sigma}-1}} - \frac{1}{(k')^{\bar{\sigma}-1}} \right)
\end{equation}
for $1\leq k < k' \leq \ddelta$.
Since $\ddelta$ scales as $O(h^{-1})$, using \eqref{eq:ineq3} and Lemma \ref{le:Bk}, we obtain
\begin{align}
\label{eq:dk0}
& \lVert h^{-1} B_{1}^{-1}\cdots B_{\ddelta}^{-1} \rVert_{\infty} \leq  \frac{C}{ h (\ddelta)^{\bar{\sigma}}} \leq C,
\\[1ex]	
\label{eq:dk2}
& \sum_{k=2}^{\ddelta} \left\lVert \prod_{t=1}^{k} B_{t}^{-1} \right\rVert_{\infty}  \leq C \sum_{k=2}^{\ddelta} \left(\frac{1}{(k-1)^{\bar{\sigma}-1}} - \frac{1}{k^{\bar{\sigma}-1}} \right) \leq C \left( 1 - \frac{1}{\ddelta^{\bar{\sigma}-1}} \right) \leq C,
\end{align}
and
\begin{align}
\label{eq:dk1}
\sum_{t=k+1}^{\ddelta} \lVert (B_k^{-1}-I) \prod_{t'=k+1}^{t} B_{t'}^{-1} \rVert_{\infty}  &\leq \frac{C(k+1)^{\bar{\sigma}}}{k} \sum_{t=k+1}^{\ddelta} \left(\frac{1}{(t-1)^{\bar{\sigma}-1}} - \frac{1}{t^{\bar{\sigma}-1}} \right)
\nonumber 
\\[1ex]
& \leq \frac{C(k+1)}{k} \left( \frac{(k+1)^{\bar{\sigma}-1}}{k^{\bar{\sigma}-1}} - \frac{(k+1)^{\bar{\sigma}-1}}{\ddelta^{\bar{\sigma}-1}} \right) \leq C,
\\[1ex]
\label{eq:dk3}
\lVert h^{-1} (I-B_k^{-1}) \prod_{t=k+1}^{\ddelta} B_{t}^{-1} \rVert_{\infty} & \leq \frac{C}{k h} \frac{(k+1)^{\bar{\sigma}}}{(\ddelta)^{\bar{\sigma}}} \leq C
\end{align}
for $1\leq k\leq \ddelta-1$.
Combining \eqref{eq:dk0}-\eqref{eq:dk3}, we have from \eqref{eq:Ph-d-1} that
\begin{equation*}
\lVert P_{h,\zeta}^{-1} \rVert_{\infty} \leq \lVert Q_{h,\zeta} \rVert_{\infty} \lVert T_{h,\zeta} \rVert_{\infty}\leq C,
\end{equation*}
where $C$ only depends on $\ml$, $\zeta$ and $\bar{\sigma}$. 
This completes the proof of \eqref{est:Ph}.
\end{proof}

Before proving Theorem \ref{le:stab_dmep}, we give the discretization error of the operators $\tilde{D}_h,~D^{-}_h$ and $L_h$. 
Given $\varphi\in C^2\big([0,1];\R^N\big) \cap \U$, we see from \eqref{def:Dh} and \eqref{def:DhL} that there exists a constant $C>0$ independent of $h$, such that
\begin{align}
\label{est:Dh}
|\tilde{D}_h\Pi_h\varphi - \Pi_h\varphi'|_{\infty} + |D^{-}_h\Pi_h\varphi - \Pi_h\varphi'|_{\infty} \leq C h.
\end{align}
This together with the standard error estimate of the numerical integration leads to
\begin{equation}
\label{est:Lh1}
|L_h(\Pi_h\varphi) - L(\varphi)| 
\leq \Big| h \sum_{t=1}^{M} \big( |(D^{-}_h \Pi_h\varphi)_t| - |\varphi'(\alpha_t)|\big)\Big| + \Big| h \sum_{t=1}^{M} \varphi'(\alpha_t)| -\int_0^1 |\varphi'(s)|\dd s \Big| 
\leq Ch.
\end{equation}
We obtain by a similar argument that
\begin{align}
\label{est:Lh2}
|\delta L_h(\Pi_h \varphi)\tdmv - \delta L(\varphi) (P_h\tdmv) | \leq Ch\qquad\forall~\tdmv\in\X_h,
\end{align}
where $P_h$ is defined in \eqref{def:Ph}.

Now we are ready to show the stability of the discrete MEP.

\begin{proof}
[{\bf Proof of Theorem \ref{le:stab_dmep}}]
For any $\tdmv\in\X_h$, we have
\begin{equation}
\label{eq:stab0}
\delta\mathscr{F}_h(\Pi_h \mv) \tdmv = \big( \delta {\mathscr{F}}_h (\Pi_h \mv) \tdmv - \Pi_h\delta\mathscr{F}(\mv)(P_h\tdmv)\big) + \Pi_h\delta\mathscr{F}(\mv)(P_h\tdmv).
\end{equation}
By comparing $\delta\mathscr{F}_h(\Pi_h\mv)\tdmv$ in \eqref{def:dFh} and $\delta\mathscr{F}(\mv)(P_h\tdmv)$ in \eqref{def:dF}, we can obtain from \eqref{est:Dh}, \eqref{est:Lh1} and \eqref{est:Lh2} that 
\begin{equation}
\label{eq:stab1}
\lVert  \delta {\mathscr{F}}_h (\Pi_h \mv) \tdmv - \Pi_h\delta\mathscr{F}(\mv)(P_h\tdmv) \rVert_{\Y_h} \leq C h \lVert P_h\tdmv \rVert_{\X} \leq C h \lVert \tdmv \rVert_{\X_h},
\end{equation}
where $C$ depends only on $E$ and $\mv$.

Now we turn to the second term in \eqref{eq:stab0}.
For any $\tdmv\in\X_h$, let $\hat{\tv}_h = P_h\tdmv$ and $f_h = \Pi_h\delta\mathscr{F}(\mv)\hat{\tv}_h$. With the representation \eqref{eq:tdmv}, using Lemma \ref{le:Bk} and Lemma \ref{le:Ph-d}, we have from \eqref{eq:drp_perp1}, \eqref{eq:drp_perp2}, \eqref{eq:drp_perp3} and \eqref{eq:drp_perp4} that
\begin{equation*}
| (\tilde{D}_h\drp^\perp)_{k} |_{\infty} \leq  C\lVert F^\perp \rVert_{\Y_h}\qquad{\rm for}~0\leq k\leq \ds+1,
\end{equation*}
where $\drp^\perp$ and $F^\perp$ are given by \eqref{def:drpperp} and \eqref{def:Fperp}, respectively.
A similar estimate holds for $\ds+2\leq k\leq M$ and then we have 
\begin{equation}
\label{est:drp_perp}
| \tilde{D}_h\drp^\perp |_{\infty} \leq C\lVert F^\perp \rVert_{\Y_h} \leq  C\lVert f_h \rVert_{\Y_h},
\end{equation}
where $C$ depends only on $\ml$ and $A$.
We obtain by taking $k=1,\cdots,\ds-1,\ds+1,\cdots,M-1$ in \eqref{eq:drp0} that
\begin{equation}
\label{eq:drp0-3}
| \tilde{D}_h\drp_0 |_{\infty} \leq \| F_{0} \|_{\Y_h} + |\drp^\perp|_{\infty} \sum_{t=1}^{M-1} \sum_{j=1}^{N-1} | c_{h,t,j} | + |\drp_0 |_{\infty} \sum_{t=1}^{M-1}| c_{h,t,0} | ,
\end{equation}
where $\drp_0$, $F_0$ and $c_{h,t,i}$ are given by \eqref{def:drpperp}, \eqref{def:Fperp} and \eqref{def:c_hti}, respectively.
From the definition of $g_{h,t,i}$ in \eqref{eq:bas} and $\hat{\X}_h$ in \eqref{def:hatXh}, we know that supp $(g_{h,t,i}) \subset [\alpha_{t-2},\alpha_{t+2}]$ for $1 \leq t\leq  M-1,~0\leq i\leq N-1$. Since $\evt_0\in C^2\big([0,1];\R^N\big)$ and $|\evt_0(\alpha)|^2 \equiv 1$, we have
\begin{align}
\label{eq:drp0-1}
|c_{h,t,0}| & \leq 4h \| \evt_0(\alpha_t)^{\rm T}\evt'_0(\alpha) \|_{\Cz([\alpha_{t-2},\alpha_{t+2}];\R)} 
\leq C h^2 \qquad{\rm and}
\\[1ex] 
\label{eq:drp0-2}
|c_{h,t,j}| & \leq 4 h \left\| g_{h,t,j}^{\rm T}\evt'_0 \right\|_{\Cz([\alpha_{t-2},\alpha_{t+2}];\R)} \leq C h,
\end{align}
where $C$ depends only on $\{ \evt_i\}_{i=0}^{N-1}$.
Combining \eqref{est:drp_perp}-\eqref{eq:drp0-2}, we have that for sufficiently small $h$,
\begin{equation}
\label{est:drp0}
\lvert \tilde{D}_h\drp_0\rvert_{\infty} 
\leq C \lvert \drp^\perp \rvert_{\infty} + C \lVert F_{0} \rVert_{\Y_h}  \leq C \lVert f_h \rVert_{\Y_h}.
\end{equation}
From \eqref{eq:tdmv} we have
\begin{align*}
(D^{-}_h\tdmv)_k 
& = h^{-1} \sum_{i=0}^{N-1} \big( \drp_{k,i}\evt_i(\alpha_k) - \drp_{k-1,i}\evt_i(\alpha_{k-1}) \big)
\\[1ex]
&= \sum_{i=0}^{N-1} \left( \frac{\drp_{k,i}-\drp_{k-1,i}}{h} \evt_i(\alpha_k) + \drp_{k-1,i}\frac{\evt_i(\alpha_k) - \evt_i(\alpha_{k-1}) }{h} \right)\qquad{\rm for}~1\leq k\leq M.
\end{align*}
This together with \eqref{est:drp_perp} and \eqref{est:drp0} implies that for sufficient small $h$, there exists a $C>0$ depending only on $\ml,~A$ and $\{\evt_i\}_{i=0}^{N-1}$ such that
\begin{equation*}
\lVert \tdmv \rVert_{\X_h} \leq C \left| \begin{pmatrix}
    \tilde{D}_h\drp^\perp \\[1ex]
    \tilde{D}_h\drp_{0}
    \end{pmatrix} \right|_{\infty} \leq C \lVert f_h \rVert_{\Y_h}.
\end{equation*}
This indicates
\begin{equation}
\label{eq:stab2}
\|\tdmv\|_{\X_h} \leq C \|\Pi_h \delta\mathscr{F}(\mv) (P_h\tdmv) \|_{\Y_h}\qquad\forall~\tdmv\in\X_h.
\end{equation}
Combining \eqref{eq:stab0}, \eqref{eq:stab1} and \eqref{eq:stab2} completes the proof of \eqref{lemme:stability}.
\end{proof}

\subsection{Convergence analysis}
\label{sec:conv}

We first show in the following theorem that the solutions of \eqref{pb:dmep} converge with respect to the mesh size $h$, and then prove Theorem \ref{thm2:conv} by connecting \eqref{dMEP:NEB} and \eqref{pb:dmep}.

\begin{theorem}
\label{thm:conv}
Let $\mv\in C^3\big( [0,1];\R^{N} \big)$ be a solution of \eqref{mep} and $\Pi_h$ be defined in \eqref{Pih}.
Assume that {\bf (A)} and {\bf (B)} are satisfied.
Then for sufficiently small $h$, there exist a solution $\mv_h \in \U_h$ of \eqref{pb:dmep} and a constant $C$ depending only on $\mv$ and $E$, such that
\begin{align}
\label{theorem:convergence-rate}
\lVert \mv_h -  \Pi_h \mv \rVert_{\X_h} \leq C h .
\end{align}
\end{theorem}

\begin{proof}
The proof is divided into three parts. 
We will first show the consistency and stability of the discrete MEP equation \eqref{pb:dmep}, and then derive the convergence by inverse function theorem \cite[Lemma 2.2]{2011qnl}. 
    
1. {\it Consistency}. 
By comparing \eqref{def:F} and \eqref{def:Fh}, we have from \eqref{est:Dh} and \eqref{est:Lh1} that
\begin{align}
\label{eq:cons}	
\lVert\mathscr{F}_h( \Pi_h \mv) \rVert_{\Y_h}
  = \lVert {\mathscr{F}}_h(\Pi_h \mv) - \Pi_h\mathscr{F}(\mv) \rVert_{\Y_h} \leq C h,
\end{align}
where $C$ only depends on $E$ and $\mv$.
    
2. {\it Stability.} 
It follows directly from Theorem \ref{le:stab_dmep} that there exists a constant $C$ depending on $E$ and $\mv$ such that
\begin{equation}
\label{eq:stab_dmep}
\lVert \delta\mathscr{F}_h(\Pi_h\mv) \tdmv\rVert_{\Y_h} \geq C \lVert \tdmv \rVert_{\X_h} \qquad \forall~\tdmv\in\X_h.
\end{equation}

3. {\it Application of the inverse function theory.}
With the consistency \eqref{eq:cons} and stability \eqref{eq:stab_dmep}, we can apply the inverse function theorem \cite[Lemma 2.2]{2011qnl} to obtain that, for $h$ sufficiently small, the solution $\mv_h$ of \eqref{pb:dmep} exists and 
\begin{equation*}
	\lVert \mv_h - \Pi_h \mv \rVert_{\X_h} \leq  C h . \qedhere
\end{equation*}
\end{proof}

Now we can prove our main result Theorem \ref{thm2:conv}, by showing the equivalence of the solutions of \eqref{dMEP:NEB} and \eqref{pb:dmep}.

\begin{proof}
[{\bf Proof of Theorem \ref{thm2:conv}}]
Let $\mv_h$ be a solution of \eqref{pb:dmep}. 
Note that the saddle $y_{\rm S}$ reaches the energy maximum along the MEP.
Then we have from the convergence result in Theorem \ref{thm:conv} that for sufficiently small $h$,
\begin{equation*}
    E(\mv_{h,0}) < E(\mv_{h,1}) < \cdots <E(\mv_{h,\ds})> E(\mv_{h,\ds+1})>\cdots>E(\mv_{h,M}) .
\end{equation*}
By comparing \eqref{eq:upwind} and \eqref{def:Dh}, we see that these two first-order difference schemes of $\mv_h$ are exactly the same, i.e., $\tilde{D}_h\mv_h = D_h\mv_h$. 
This together with \eqref{eq:eqdis} yields that $\mv_h$ is a solution of \eqref{dMEP:NEB} and it remains to show the estimates \eqref{est2:dismep} and \eqref{est:Eb}.

From \eqref{est:Dh} and Theorem \ref{thm:conv}, we estimate
\begin{align*}
& \max_{0\leq k\leq M}\Big( |(D_h\mv_h)_k - \mv'(\alpha_k)| + |\mv_{h,k} - \mv(\alpha)| \Big)
\\[1ex]
& \qquad \leq C \big( |\tilde{D}_h\mv_h - \Pi_h\mv'|_{\infty} + |\mv_h - \Pi_h\mv|_{\infty} \big)
\\[1ex]
& \qquad \leq C |\tilde{D}_h(\mv_h - \Pi_h\mv)|_{\infty} +  C | \tilde{D}_h\Pi_h\mv - \Pi_h\mv'|_{\infty} + C|\mv_h - \Pi_h\mv|_{\infty}
\\[1ex]
& \qquad \leq C |D^{-}_h( \mv_h - \Pi_h\mv) |_{\infty} + Ch + C|\mv_h - \Pi_h\mv|_{\infty} \leq Ch ,
\end{align*}
which completes the proof of \eqref{est2:dismep}.

Finally, we turn to the error estimate for the energy barrier. 
Using the fact that $\nabla E\big(\mv(\sad)\big) = 0$ and $E\in C^4(\R^N)$, we have
\begin{equation}
\label{eq:err_Eb1}
\big| \Delta E(\mv) - \Delta_h E(\mv_h)\big| = \big| E\big(\mv(\sad)\big) - E(\bar{\varphi}_{h,\ds})\big| 
\leq C | \mv(\sad) - \mv_{h,\ds})|_{\infty}^2 ,
\end{equation}
where $C$ is a constant depending only on $E$ and $\mv$.
Note that Theorem \ref{thm:conv} and the fact that $\mv\in C^3\big([0,1];\R^N\big)$ implies
\begin{equation*}
| \mv(\sad) - \mv_{h,\ds})|_{\infty} \leq | \mv(\sad) - \mv(\alpha_{\ds})|_{\infty} + | \mv(\alpha_{\ds}) - \mv_{h,\ds}) |_{\infty} \leq C h,
\end{equation*}
which together with \eqref{eq:err_Eb1} yields \eqref{est:Eb} and completes the proof.
\end{proof}

\small
\bibliographystyle{plain}
\bibliography{bib}

\begin{thebibliography}{10}

\bibitem{berglund2013kramers}
N.~Berglund.
\newblock Kramers' law: Validity, derivations and generalisations.
\newblock {\em Markov Process. Related Fields}, 19:459--490, 2013.

\bibitem{Braun2020SharpUC}
J.~Braun and C.~Ortner.
\newblock Sharp uniform convergence rate of the supercell approximation of a
  crystalline defect.
\newblock {\em SIAM J. Numer. Anal.}, 58:279--297, 2020.

\bibitem{cameron2011string}
M.~Cameron, R.~Kohn, and E.~Vanden-Eijnden.
\newblock The string method as a dynamical system.
\newblock {\em J. Nonlinear Sci.}, 21:193--230, 2011.

\bibitem{2002string}
W.~E, W.~Ren, and E.~Vanden-Eijnden.
\newblock String method for the study of rare events.
\newblock {\em Phys. Rev. B}, 66:052301, 2002.

\bibitem{2007string}
W.~E, W.~Ren, and E.~Vanden-Eijnden.
\newblock Simplified and improved string method for computing the minimum
  energy paths in barrier-crossing events.
\newblock {\em J. Chem. Phys.}, 126:164103, 2007.

\bibitem{hanggi1990reaction}
P.~H{\"a}nggi, P.~Talkner, and M.~Borkovec.
\newblock Reaction-rate theory: fifty years after kramers.
\newblock {\em Rev. Modern Phys.}, 62:251--341, 1990.

\bibitem{2000climbingNEB}
G.~Henkelman, B.~P. Uberuaga, and H.~J{\'o}nsson.
\newblock A climbing image nudged elastic band method for finding saddle points
  and minimum energy paths.
\newblock {\em J. Chem. Phys.}, 113:9901--9904, 2000.

\bibitem{jonsson1998NEB}
H.~J{\'o}nsson, G.~Mills, and K.~W. Jacobsen.
\newblock Nudged elastic band method for finding minimum energy paths of
  transitions.
\newblock In {\em Classical and Quantum Dynamics in Condensed Phase
  Simulations}, pages 385--404. Citeseer, 1998.

\bibitem{2018LuskinStabString}
B.~Koten and M.~Luskin.
\newblock Stability and convergence of the string method for computing minimum
  energy paths.
\newblock {\em Multiscale Model. Simul.}, 17:873--898, 2018.

\bibitem{MAP}
G.~D. Leines and J.~Rogal.
\newblock Comparison of minimum-action and steepest-descent paths in gradient
  systems.
\newblock {\em Phys. Rev. E}, 93:022307, 2016.

\bibitem{MEPStab}
X.~Liu, H.~Chen, and C.~Ortner.
\newblock Stability of the minimum energy path.
\newblock {\em Numer. Math.}, 156:39--70, 2024.

\bibitem{liu2022stability}
Xuanyu Liu, Huajie Chen, and Christoph Ortner.
\newblock Stability of the minimum energy path, 2022.

\bibitem{muller1979MEP}
K.~M{\"u}ller and L.~D. Brown.
\newblock Location of saddle points and minimum energy paths by a constrained
  simplex optimization procedure.
\newblock {\em Theor. Chim. Acta}, 53:75--93, 1979.

\bibitem{2011qnl}
C.~Ortner.
\newblock A priori and a posteriori analysis of the quasinonlocal
  quasicontinuum method in 1d.
\newblock {\em Math. Comp.}, 80:1265–1285, 2011.

\bibitem{JuLIP}
C.~Ortner, J.~Kermode, et~al.
\newblock {JuLIP: Julia Library for Interatomic Potentials}.
\newblock \url{https://github.com/JuliaMolSim/JuLIP.jl}.

\bibitem{2013climbingstring}
W.~Ren and E.~Vanden-Eijnden.
\newblock A climbing string method for saddle point search.
\newblock {\em J. Chem. Phys.}, 138:134105, 2013.

\bibitem{vineyard1957frequency}
G.~H. Vineyard.
\newblock Frequency factors and isotope effects in solid state rate processes.
\newblock {\em J. Phys. Chem. Solids}, 3:121--127, 1957.

\end{thebibliography}

\end{document}